\documentclass[11pt]{article}
\usepackage{amsfonts,latexsym,rawfonts,amsmath,amssymb,amsthm}
\usepackage{lscape}
\usepackage[cm]{fullpage}
\usepackage{amscd, float,times,rotating}
\usepackage{pb-diagram}
\usepackage{hyperref}
\usepackage{amsthm}
\usepackage{enumerate}
\usepackage[latin1]{inputenc}
\usepackage[T1]{fontenc}
\usepackage{mathrsfs}
\usepackage{amscd}
\usepackage{amssymb}
\usepackage[english]{babel}
\usepackage{graphicx}

\usepackage{geometry}
\newtheorem{theorem}{Theorem}[section]
\newtheorem{lemma}[theorem]{Lemma}

\newtheorem{proposition}[theorem]{Proposition}

\theoremstyle{definition}
\newtheorem{definition}[theorem]{Definition}

\newtheorem{remark}[theorem]{Remark}

\numberwithin{equation}{section}

\title{Quantitative stratification of $F$-subharmonic functions}

\author{Jianchun Chu}

\begin{document}

\date{}

\maketitle

\begin{abstract}
In this paper, we study the singular sets of $F$-subharmonic functions $u: B_{2}(0^{n})\rightarrow\mathbf{R}$, where $F$ is a subequation. The singular set $\mathcal{S}(u)\subset B_{2}(0^{n})$ has a stratification $\mathcal{S}^{0}(u)\subset\mathcal{S}^{1}(u)\subset\cdots\subset\mathcal{S}^{k}(u)\subset\cdots\subset\mathcal{S}(u)$, where $x\in\mathcal{S}^{k}(u)$ if no tangent function to $u$ at $x$ is $(k+1)$-homogeneous. We define the quantitative stratification $\mathcal{S}_{\eta,r}^{k}(u)$ and $\mathcal{S}_{\eta}^{k}(u)=\cap_{r}\mathcal{S}_{\eta,r}^{k}(u)$.

When homogeneity of tangents holds for $F$, we prove that $dim_{H}\mathcal{S}^{k}(u)\leq k$ and $\mathcal{S}(u)=\mathcal{S}^{n-p}(u)$, where $p$ is the Riesz characteristic of $F$. And for the top quantitative stratification $\mathcal{S}_{\eta}^{n-p}(u)$, we have the Minkowski estimate $\text{Vol}(B_{r}(\mathcal{S}_{\eta}^{n-p}(u)\cap B_{1}(0^{n})))\leq C\eta^{-1}(\int_{B_{1+r}(0^{n})}\Delta u)r^{p}$.

When uniqueness of tangents holds for $F$, we show that $S_{\eta}^{k}(u)$ is $k$-rectifiable, which implies $\mathcal{S}^{k}(u)$ is $k$-rectifiable.

When strong uniqueness of tangents holds for $F$, we introduce the monotonicity condition and the notion of $F$-energy. By using refined covering argument, we obtain a definite upper bound on the number of $\{\Theta(u,x)\geq c\}$ for $c>0$, where $\Theta(u,x)$ is the density of $F$-subharmonic function $u$ at $x$.

Geometrically determined subequations $F(\mathbb{G})$ is a very important kind of subequation (when $p=2$, homogeneity of tangents holds for $F(\mathbb{G})$; when $p>2$, uniqueness of tangents holds for $F(\mathbb{G})$). By introducing the notion of $\mathbb{G}$-energy and using quantitative differentation argument, we obtain the Minkowski estimate of quantitative stratification $\text{Vol}(B_{r}(\mathcal{S}_{\eta,r}^{k}(u))\cap B_{1}(0^{n}))\leq Cr^{n-k-\eta}$.
\end{abstract}

\tableofcontents

\bigskip

\section{Introduction}
\subsection{Background}
Recently, Harvey and Lawson \cite{HL1,HL2} (see also \cite{HL5,HL6,HL7,HL8,HL9,HL10,HL11,HL12,HL3,HL13,HL4,HL14}) established a theory of elliptic equations. The aim of this theory is to study the behavior of subsolutions in the viscosity sense. They introduced the definitions of Riesz characteristic, tangential $p$-flow, tangent and density function. And many interesting theorems, formulas and properties of subsolutions, tangents and density functions were established.

In this theory, there is a very important kind of examples called geometrically defined subequations (see \cite[Example 4.4]{HL1} and \cite{HL2}). To be specific, let $\mathbb{G}$ be a compact subset of the Grassmannian manifold $G(p,\mathbf{R}^{n})$ such that $\mathbb{G}$ is invariant under a subgroup $G\subset O(n)$ acting transitively on the sphere $S^{n-1}\subset\mathbf{R}^{n}$. The geometric subequation determined by $\mathbb{G}$ is defined by
\begin{equation*}
F(\mathbb{G})=\{A\in \text{Sym}(n)~|~tr_{W}(A)\geq 0 \text{~for any $W\in\mathbb{G}$}\},
\end{equation*}
where $\text{Sym}(n)$ denotes the space of symmetric $n\times n$ matrices with real entries and $tr_{W}(A)$ denotes the trace of $A|_{W}$. Let $u$ be a $F(\mathbb{G})$-subharmonic function, by the Restriction Theorem 3.2 in \cite{HL4}, we obtain $u|_{W}$ is subharmonic on $W$ for any $W\in\mathbb{G}$. $F(\mathbb{G})$-subharmonic functions are usually called $\mathbb{G}$-plurisubharmonic functions. And as we can see, convex, $\mathbb{C}$-plurisubharmonic and $\mathbb{H}$-plurisubharmonic functions are all special cases of $\mathbb{G}$-plurisubharmonic functions.

In \cite{HL1}, Harvey and Lawson introduced the definitions of homogeneity, uniqueness and strong uniqueness of tangents. In \cite{HL2}, for geometrically defined subequations $F(\mathbb{G})$, it was proved that homogeneity of tangents holds when $p=2$ and uniqueness of tangents holds when $p>2$. They also proved strong uniqueness of tangents holds for many subequations (see \cite[Theorem 13.1]{HL1} and \cite[Theorem 3.2, Theorem 3.12]{HL2}). When the subequation $F$ is convex, for any $F$-subharmonic function $u$, upper semicontinuity of density functions $\Theta^{M}(u,\cdot)$, $\Theta^{S}(u,\cdot)$ and $\Theta^{V}(u,\cdot)$ was proved (see \cite[Theorem 7.4]{HL1}), which implies that for any $c>0$ and each density function as above, the set
\begin{equation*}
E_{c}(u):=\{x~|~\Theta(u,x)\geq c\}
\end{equation*}
is closed (see \cite[Corollary 7.5]{HL1}). Furthermore, the discreteness of the set $E_{c}(u)$ was established when strong uniqueness of tangents holds for $F$ and $p>2$, where $p$ is the Riesz characteristic of $F$ (see \cite[Theorem 14.1, Theorem 14.1']{HL1}).

\subsection{Definitions and notations}
In this paper, many definitions in Harvey and Lawson's theory will be used. For these details, we refer the reader to \cite{HL1,HL2}. We shall use the following notations, for any function $u$, point $x\in\mathbf{R}^{n}$ and $r>0$,
\begin{equation*}
M(u,x,r)=\sup_{y\in B_{1}(0^{n})}u(x+ry),
\end{equation*}
\begin{equation*}
S(u,x,r)=\frac{1}{n\omega_{n}}\int_{\partial B_{1}(0^{n})}u(x+ry)dy,
\end{equation*}
\begin{equation*}
V(u,x,r)=\frac{1}{\omega_{n}}\int_{B_{1}(0^{n})}u(x+ry)dy,
\end{equation*}
where $0^{n}$ is the origin in $\mathbf{R}^{n}$ and $\omega_{n}$ is the volume of unit ball in $\mathbf{R}^{n}$.

Let $F$ be a subequation satisfying Positivity, ST-Invariance, Cone Property and Convexity (see \cite[p.2-3]{HL1}). We assume that the Riesz characteristic of $F$ is $p$ (see \cite[Definition 3.2]{HL1}). First, let us recall some definitions in \cite{HL1}.

\begin{definition}(\cite[Definition 9.1]{HL1})\label{definition of tangential p-flow}
Suppose that $u$ is a $F$-subharmonic function. Let $x$ be a point such that $B_{\rho}(x)$ is in the domain of $u$, where $\rho>0$. For any $r>0$, the tangential $p$-flow (or $p$-homothety) of $u$ at $x$ is defined as follows.
\begin{enumerate}[~~~~~~(1)]
    \item If $p>2$, $u_{x,r}(y):=r^{p-2}u(x+ry)$ in $B_{\frac{\rho}{r}}(0^{n})$;
    \item If $2>p\geq 1$, $u_{x,r}(y):=\frac{1}{r^{2-p}}\left(u(x+ry)-u(x)\right)$ in $B_{\frac{\rho}{r}}(0^{n})$;
    \item If $p=2$, $u_{x,r}(y):=u(x+ry)-M(u,x,r)$ in $B_{\frac{\rho}{r}}(0^{n})$.
\end{enumerate}
\end{definition}

\begin{definition}(\cite[Definition 12.1]{HL1})\label{three definitions}
Suppose that $u$ is a $F$-subharmonic function. Let $T_{x}(u)$ be the tangent set to $u$ at $x$ (see \cite[Definition 9.3]{HL1}), where $x$ is a interior point in the domain of $u$.
\begin{enumerate}[~~~~~~(1)]
    \item For any $u$ and $x$, if every tangent $\varphi\in T_{x}(u)$ satisfies $\varphi_{0^{n},r}=\varphi$ for any $r>0$, we say that homogeneity of tangents holds for $F$;
    \item For any $u$ and $x$, if $T_{x}(u)$ is a singleton, we say that uniqueness of tangents holds for $F$;
    \item For any $u$ and $x$, if $T_{x}(u)=\{\Theta K_{p}(|\cdot|)\}$, where $\Theta\geq0$ is a constant and $K_{p}$ is the classical $p^{th}$ Riesz kernel (see \cite[(1.1)]{HL1}), we say that strong uniqueness of tangents holds for $F$.
\end{enumerate}
\end{definition}

\begin{remark}
In Definition \ref{three definitions}, it is clear that (3) implies (2) and (2) implies (1).
\end{remark}

Next, in order to study the singular sets of $F$-subharmonic functions, we have the following definitions.

\begin{definition}\label{homogeneous definition}
A function $h:\mathbf{R}^{n}\rightarrow\mathbf{R}$ is said to be $k$-homogeneous at $x\in\mathbf{R}^{n}$ with respect to $k$-plane $V^{k}\subset\mathbf{R}^{n}$ if $h$ satisfies the following properties:
\begin{enumerate}[~~~~~~(1)]
    \item $h$ is subharmonic on $\mathbf{R}^{n}$;
    \item For any $r>0$, $h_{x,r}(y)=h(y+x)$ for every $y\in \mathbf{R}^{n}$, where $h_{x,r}$ is the tangential $p$-flow of $h$ at $x$;
    \item For any $y\in\mathbf{R}^{n}$ and $v\in V^{k}$, $h(y+v+x)=h(y+x)$.
\end{enumerate}
If $x=0^{n}$, we say $h$ is $k$-homogeneous (or $h$ is a $k$-homogeneous function) for convenience.
\end{definition}

\begin{definition}
A function $u:B_{2r}(x)\subset\mathbf{R}^{n}\rightarrow\mathbf{R}$ is said to be $(k,\epsilon,r,x)$-homogeneous, if there exists a $k$-homogeneous function $h:\mathbf{R}^{n}\rightarrow\mathbf{R}$ such that
\begin{equation*}
\|u_{x,r}-h\|_{L^{1}(B_{1}(0^{n}))}<\epsilon.
\end{equation*}
\end{definition}

\begin{definition}
Suppose that homogeneity of tangents holds for $F$. Let $u$ be a $F$-subharmonic function on $B_{2}(0^{n})$. For any $\eta>0$ and $r\in(0,1)$, we have the following definitions
\begin{enumerate}[~~~~~~(1)]
    \item The singular set $\mathcal{S}(u)$ is defined by
\begin{equation*}
\mathcal{S}(u):=\{x\in B_{2}(0^{n})~|~\text{no tangent at $x$ is $n$-homogeneous}\}.
\end{equation*}
    \item The $k^{th}$ stratification $\mathcal{S}^{k}(u)$ is defined by
\begin{equation*}
\mathcal{S}^{k}(u):=\{x\in B_{2}(0^{n})~|~\text{no tangent at $x$ is $(k+1)$-homogeneous}\}.
\end{equation*}
    \item The $k^{th}$ $\eta$-stratification $\mathcal{S}_{\eta}^{k}(u)$ is defined by
\begin{equation*}
\mathcal{S}_{\eta}^{k}(u):=\{x\in B_{2}(0^{n})~|~\text{$u$ is not $(k+1,\eta,s,x)$-homogeneous for any $s\in(0,1)$}\}.
\end{equation*}
    \item The $k^{th}$ $(\eta,r)$-stratification $\mathcal{S}_{\eta,r}^{k}(u)$ is defined by
\begin{equation*}
\mathcal{S}_{\eta,r}^{k}(u):=\{x\in B_{2}(0^{n})~|~\text{$u$ is not $(k+1,\eta,s,x)$-homogeneous for any $s\in[r,1)$}\}.
\end{equation*}
\end{enumerate}
\end{definition}

\begin{remark}
When homogeneity of tangents holds for $F$, we have the following relationships (see Proposition \ref{quantitative stratification proposition})
\begin{equation*}
\mathcal{S}^{0}(u)\subset\mathcal{S}^{1}(u)\subset\cdots\subset\mathcal{S}^{n-1}(u)=\mathcal{S}(u)
\end{equation*}
and
\begin{equation}\label{quantitative stratification}
\mathcal{S}^{k}(u)=\bigcup_{\eta}\mathcal{S}_{\eta}^{k}(u)=\bigcup_{\eta}\bigcap_{r}\mathcal{S}_{\eta,r}^{k}(u).
\end{equation}
\end{remark}

\begin{remark}
When strong uniqueness of tangents holds for $F$, three density functions $\Theta^{M}(u,\cdot)$, $\Theta^{S}(u,\cdot)$ and $\Theta^{V}(u,\cdot)$ are equivalent (see \cite[Proposition 7.1, (12.3)]{HL1}). And for each density function as above, we have
\begin{equation*}
\mathcal{S}(u)=\mathcal{S}^{0}(u)=\bigcup_{c>0}E_{c}(u),
\end{equation*}
where $E_{c}(u)=\{x\in B_{2}(0^{n})~|~\Theta(u,x)\geq c\}$.
\end{remark}

\subsection{Main results}
In this paper, we assume that $F$ is a subequation satisfies Positivity, ST-Invariance, Cone Property and Convexity (see \cite[p.2-3]{HL1}). Let $p$ be the Riesz characteristic of $F$ (see \cite[Definition 3.2]{HL1}). When $1\leq p<2$, the $F$-subharmonic function is H\"{o}lder continuous (see \cite[Theorem 15.1]{HL1}). Hence, we focus on the case $p\geq2$ in this paper. Our main results are the following theorems.

\begin{theorem}\label{first main result}
Suppose that $F$ is a subequation such that homogeneity of tangents holds for $F$. Let $u$ be a $F$-subharmonic function defined on $B_{2}(0^{n})$ with $\|u\|_{L^{1}(B_{2}(0^{n}))}\leq\Lambda$. For any $\eta>0$, we have
\begin{enumerate}[~~~~~~(1)]
    \item $\text{Vol}(B_{r}(S_{\eta}^{n-p}(u)\cap B_{1}(0^{n})))\leq C(p,n)\eta^{-1}\left(\int_{B_{1+r}(0^{n})}\Delta u\right)r^{p}$ for any $r\in(0,\frac{1}{5})$;
    \item $\mathcal{S}(u)=\mathcal{S}^{n-p}(u)$;
    \item $dim_{H}(\mathcal{S}^{k}(u))\leq k$ for any $k=1,2,\cdots,n$, where $dim_{H}\mathcal{S}^{k}(u)$ is the Hausdorff dimension of $\mathcal{S}^{k}(u)$.
\end{enumerate}
\end{theorem}

\begin{theorem}\label{Rectifiability theorem}
Suppose that $F$ is a subequation such that uniqueness of tangents holds for $F$. Let $u$ be a $F$-subharmonic function defined on $B_{2}(0^{n})$. Then $\mathcal{S}^{k}(u)$ is $k$-rectifiable for any $k=1,2,\cdots,n$.
\end{theorem}

\begin{theorem}\label{estimate of singular set}
Suppose that $F$ is a subequation such that strong uniqueness of tangents holds for $F$ and $p>2$. Let $u$ be a $F$-subharmonic function defined on $B_{2}(0^{n})$ with $\|u\|_{L^{1}(B_{2}(0^{n}))}\leq\Lambda$. For any $c>0$, there exists a constant $C(c,\Lambda,F)$ such that
\begin{equation}\label{estimate of singular set equation3}
\#\left(E_{c}(u)\cap B_{1}(0^{n})\right)\leq C(c,\Lambda,F).
\end{equation}
\end{theorem}

In the proof of Theorem \ref{estimate of singular set}, we introduce the monotonicity condition and the notion of $F$-energy. And we prove every $F$-subharmonic function satisfies monotonicity condition after subtracting a constant. For $F$-subharmonic function satisfies monotonicity condition, we prove (\ref{estimate of singular set equation3}) by using refined covering arguments, which is introduced in \cite{NVV}. Since the set $E_{c}(u)$ is invariant after subtracting a constant, we obtain Theorem \ref{estimate of singular set}.

For geometrically defined subequations $F(\mathbb{G})$ (i.e., $\mathbb{G}$-plurisubharmonic case), we have the following Minkowski estimate of quantitative stratification.
\begin{theorem}\label{main theorem for G-plurisubharmonic}
Let $u$ be a $\mathbb{G}$-plurisubharmonic function on $B_{2}(0^{n})$ with $\|u\|_{L^{1}(B_{2}(0^{n}))}\leq\Lambda$. For any $\eta>0$, there exists constant $C(\eta,\Lambda,\mathbb{G})$ such that for any $r\in(0,1)$, we have
\begin{equation}\label{main theorem for G-plurisubharmonic equation}
\text{Vol}(B_{r}(\mathcal{S}_{\eta,r}^{k}(u))\cap B_{1}(0^{n}))\leq C(\eta,\Lambda,\mathbb{G})r^{n-k-\eta}.
\end{equation}
\end{theorem}

\begin{remark}\label{smooth remark}
It suffices to prove Theorem \ref{main theorem for G-plurisubharmonic} when $\mathbb{G}$ is a smooth submanifold of $G(p,\mathbf{R}^{n})$. For general $\mathbb{G}$, since $\mathbb{G}$ is invariant under a subgroup $G\subset O(n)$ acting transitively on the sphere $S^{n-1}\subset\mathbf{R}^{n}$, we fix $W\in\mathbb{G}$ and consider $\mathbb{G}_{0}=G\cdot W$. Then $\mathbb{G}_{0}$ is a smooth submanifold of $G(p,\mathbf{R}^{n})$ and $F(\mathbb{G})\subset F(\mathbb{G}_{0})$. It follows that any $\mathbb{G}$-plurisubharmonic function is $\mathbb{G}_{0}$-plurisubharmonic function. Then Theorem \ref{main theorem for G-plurisubharmonic} for smooth $\mathbb{G}_{0}$ implies Theorem \ref{main theorem for G-plurisubharmonic} for general $\mathbb{G}$ (see \cite[p.9]{HL2}). Therefore, without loss of generality, we assume that $\mathbb{G}$ is a smooth submanifold of $G(p,\mathbf{R}^{n})$ in Section 7.
\end{remark}

In the proof of Theorem \ref{main theorem for G-plurisubharmonic}, we introduce the notion of $\mathbb{G}$-energy, which is a monotone quantity. The key point is to establish the quantitative rigidity theorem (Theorem \ref{quantitative rigidity theorem, p>2} and Theorem \ref{quantitative rigidity theorem, p=2}). Roughly speaking, we prove it by making use of the information of tangent at infinity, together with a contradiction argument. Next, combining quantitative rigidity theorem (Theorem \ref{quantitative rigidity theorem, p>2} and Theorem \ref{quantitative rigidity theorem, p=2}) and cone-splitting lemma (Lemma \ref{cone-splitting lemma}), we obtain decomposition lemma (Lemma \ref{number of covering lemma}), which implies Theorem \ref{main theorem for G-plurisubharmonic}.

In general outline, we will follow a scheme introduced in \cite{CN13a}, where quantitative differentation argument was established. By this method, Cheeger and Naber proved some new estimates on non-collapsed Riemannian manifolds with Ricci curvature bounded from below, especially Einstein manifolds. In fact, this method has already been applied to many areas. Analogous results were obtained in the study of mean curvature flows, elliptic equations, harmonic maps and so on (see \cite{CHN13a,CHN13b,CN13a,CN13b,CNV15}).

Recently, Naber and Valtorta \cite{NV14} introduced new techniques for estimating the critical and singular set of elliptic PDEs. In \cite{NV15a,NV15b,NV16b}, they also got some new results on stationary and minimizing harmonic maps. It was proved that the $k^{th}$ stratification of singular set is $k$-rectifiable and obtained more stronger estimates of the quantitative stratification. And these techniques have also been applied to the study of stationary Yang Mills (see \cite{NV16a}) and $L^{2}$ curvature bounds on non-collapsed Riemannian manifolds with bounded Ricci curvature (see \cite{JN}).

\bigskip

{\bf Acknowledgments. }The author would like to thank his advisor Professor Gang Tian for encouragement and support. The author would also like to thank Professor Aaron Naber for suggesting this problem and many helpful conversations. Partial work was done while the author was visiting the Department of Mathematics at Northwestern University, supported by the China Scholarship Council (File No. 201506010010). The author would like to thank the China Scholarship Council for supporting this visiting. The author would also like to thank the Department of Mathematics at Northwestern University for its hospitality and for providing a good academic environment.

\section{Cone-splitting lemma}
In this section, we prove cone-splitting lemma (Lemma \ref{cone-splitting lemma}) for $F$-subharmonic functions. And we will use it throughout this paper.
\begin{theorem}[Cone-splitting principle]\label{cone-splitting principle}
Let $h$ be a function which is $k$-homogeneous at $x_{1}$ with respect to $k$-plane $V^{k}$. If there exists a point $x_{2}\not\in x_{1}+V^{k}$ such that $h$ is $0$-homogeneous at $x_{2}$, then $h$ is $(k+1)$-homogeneous at $x_{1}$ with respect to $(k+1)$-plane $V^{k+1}=span\{x_{2}-x_{1},V^{k}\}$.
\end{theorem}

\begin{proof}
Let $\{e_{i}\}_{i=1}^{n}$ be the standard basis of $\mathbf{R}^{n}$. Without loss of generality, we assume that $x_{1}=0^{n}$, $x_{2}=e_{k+1}$ and $V^{k}=span\{e_{i}\}_{i=1}^{k}$. Since $h$ is $k$-homogeneous at $x_{1}$ respect to $V^{k}$, it suffices to prove
\begin{equation}\label{cone-splitting principle equation 1}
h(x+te_{k+1})=h(x),
\end{equation}
for all $x\in\mathbf{R}^{n}$ and $t\in\mathbf{R}$. We split into different cases according to $p$ (Riesz characteristic of $F$).

\bigskip
\noindent
{\bf Case 1.} $p>2$.

\bigskip
Since $h$ is $k$-homogeneous at $0^{n}$ and $0$-homogeneous at $e_{k+1}$, By the definition of homogeneous function, we have
\begin{equation*}
h(x)=h_{0^{n},\frac{1}{|x|}}(x)=|x|^{2-p}h\left(\frac{x}{|x|}\right).
\end{equation*}
Let $g_{1}=h|_{S^{n-1}}$, we obtain
\begin{equation}\label{cone-splitting principle equation 2}
h(x)=|x|^{2-p}g_{1}\left(\frac{x}{|x|}\right).
\end{equation}
Similarly, there exists function $g_{2}$ on the unit sphere $S^{n-1}\subset\mathbf{R}^{n}$ such that
\begin{equation}\label{cone-splitting principle equation 3}
h(x)=|x-e_{k+1}|^{2-p}g_{2}\left(\frac{x-e_{k+1}}{|x-e_{k+1}|}\right).
\end{equation}
We split up into different subcases.

\bigskip
\noindent
{\bf Subcase 1.1.} $x\in span\{e_{k+1}\}$.

\bigskip
By (\ref{cone-splitting principle equation 2}) and (\ref{cone-splitting principle equation 3}), we have
\begin{equation*}
2^{2-p}g_{1}(e_{k+1})=h(2e_{k+1})=g_{2}(e_{k+1})
\end{equation*}
and
\begin{equation*}
3^{2-p}g_{1}(e_{k+1})=h(3e_{k+1})=2^{2-p}g_{2}(e_{k+1}).
\end{equation*}
Hence, we obtain $g_{1}(e_{k+1})=g_{2}(e_{k+1})=0$ or $g_{1}(e_{k+1})=g_{2}(e_{k+1})=\infty$, which implies ($\ref{cone-splitting principle equation 1}$).

\bigskip
\noindent
{\bf Subcase 1.2.} $x\not\in span\{e_{k+1}\}$ and $t<1$.

\bigskip
By (\ref{cone-splitting principle equation 2}) and (\ref{cone-splitting principle equation 3}), we have
\begin{equation*}
h\left(\frac{x}{1-t}\right)=|\frac{x}{1-t}|^{2-p}g_{1}\left(\frac{x}{|x|}\right)=\frac{1}{|1-t|^{2-p}}h(x)
\end{equation*}
and
\begin{equation*}
h\left(\frac{x}{1-t}\right)=|\frac{x}{1-t}-e_{k+1}|^{2-p}g_{2}\left(\frac{\frac{x}{1-t}-e_{k+1}}{|\frac{x}{1-t}-e_{k+1}|}\right)=\frac{1}{|1-t|^{2-p}}h(x+te_{k+1}).
\end{equation*}
Then we obtain (\ref{cone-splitting principle equation 1}).

\bigskip
\noindent
{\bf Subcase 1.3.} $x\not\in span\{e_{k+1}\}$ and $t\geq1$.

\bigskip
If $x\not\in span\{e_{k+1}\}$, then $x+te_{k+1}\not\in span\{e_{k+1}\}$. By Subcase 1.2, we have $h(x)=h(x+te_{k+1}-te_{k+1})=h(x+te_{k+1})$, which implies (\ref{cone-splitting principle equation 1}).

\bigskip
\noindent
{\bf Case 2.} $p=2$.

\bigskip
By the property of homogeneous function (see \cite[p.39]{HL1}), there exists two constants $\Theta_{1},\Theta_{2}\geq 0$ and two functions $g_{1},g_{2}$ defined on the unit sphere $S^{n-1}\subset\mathbf{R}^{n}$ such that
\begin{equation*}
\begin{split}
h(x) & = \Theta_{1}\log|x|+g_{1}\left(\frac{x}{|x|}\right)\\
& = \Theta_{2}\log|x-e_{k+1}|+g_{2}\left(\frac{x-e_{k+1}}{|x-e_{k+1}|}\right).
\end{split}
\end{equation*}
First, let us prove $\Theta_{1}=\Theta_{2}$. For any point $y\not\in span\{e_{k+1}\}$ such that $h(y)>-\infty$, by similar calculations in Subcase 1.2, for any $t<1$, we obtain
\begin{equation}\label{cone-splitting principle equation 4}
h(y+te_{k+1})=h(y)+(\Theta_{2}-\Theta_{1})\log(1-t).
\end{equation}
Since $h\not\equiv-\infty$, there exists a point $x_{0}\not\in span\{e_{k+1}\}$ such that $h(x_{0})>-\infty$. By (\ref{cone-splitting principle equation 4}), we have
\begin{equation}\label{cone-splitting principle equation 5}
h(x_{0}+\frac{1}{3}e_{k+1})=h(x_{0})+(\Theta_{2}-\Theta_{1})\log\frac{2}{3}
\end{equation}
and
\begin{equation}\label{cone-splitting principle equation 6}
h(x_{0}+\frac{2}{3}e_{k+1})=h(x_{0})+(\Theta_{2}-\Theta_{1})\log\frac{1}{3}.
\end{equation}
By (\ref{cone-splitting principle equation 5}), we obtain $h(x_{0}+\frac{1}{3}e_{k+1})>-\infty$. Combining this and (\ref{cone-splitting principle equation 4}), it is clear that
\begin{equation}\label{cone-splitting principle equation 7}
h(x_{0}+\frac{1}{3}e_{k+1}+\frac{1}{3}e_{k+1})=h(x_{0}+\frac{1}{3}e_{k+1})+(\Theta_{2}-\Theta_{1})\log\frac{2}{3}.
\end{equation}
Combining (\ref{cone-splitting principle equation 5}), (\ref{cone-splitting principle equation 6}) and (\ref{cone-splitting principle equation 7}), we get $\Theta_{1}=\Theta_{2}$. Next, by the similar argument of Case 1, we obtain (\ref{cone-splitting principle equation 1}).
\end{proof}

\begin{lemma}\label{F subharmonic converge lemma}
Let $u_{i}$ be a sequence of $F$-subharmonic functions on $B_{2}(0^{n})$ with $\|u_{i}\|_{L^{1}(B_{2}(0^{n}))}\leq\Lambda$. Then there exists a subsequence $u_{i_{k}}$ such that $u_{i_{k}}$ converge to $u$ in $L^{1}_{loc}(B_{2}(0^{n}))$, where $u$ is a $F$-subharmonic function on $B_{2}(0^{n})$.
\end{lemma}

\begin{proof}
Every $F$-subharmonic function is subharmonic function (see \cite[(6.3)]{HL1}). By the compactness of subharmonic functions, there exists a subsequence $u_{i_{k}}$ converge to $u$ in $L^{1}_{loc}(B_{2}(0^{n}))$. On the other hand, $F$ is a subequation satisfying ST-Invariance and Convexity, which implies that $F$ is regular (see \cite[Section 8]{HL3}) and can not be defined using fewer of the independent variables (see \cite[Proof of Proposition 9.4]{HL1}). Since $u_{i_{k}}$ is $F$-subharmonic, we obtain that $u$ is a $F$-subharmonic distribution (see \cite[Definition 2.3]{HL3}). By \cite[Theroem 1.1]{HL3} or \cite[Theorem 9.5]{HL1}, there exists a $F$-subharmonic function $v$ in the $L^{1}_{loc}$-class $u$. For any subharmonic function $h$ on $B_{2}(0^{n})$, we have
\begin{equation*}
h(x)=\lim_{s\rightarrow\infty}\frac{1}{\omega_{n}s^{n}}\int_{B_{s}(x)}h(y)dy,
\end{equation*}
for any $x\in B_{2}(0^{n})$. Since $u$ and $v$ are subharmonic, we obtain $u=v$ in $B_{2}(0^{n})$.
\end{proof}

\begin{lemma}[Cone-splitting lemma]\label{cone-splitting lemma}
Let $u$ be a $F$-subharmonic function on $B_{2}(0^{n})$ with $\|u\|_{L^{1}(B_{2}(0^{n}))}\leq\Lambda$. For any $\epsilon,\tau>0$, there exists constant $\delta(\epsilon,\tau,\Lambda,F)$ such that if
\begin{enumerate}[~~~~~~(1)]
    \item $u$ is $(k,\delta,1,0^{n})$-homogeneous with respect to $k$-plane $V^{k}$;
    \item $u$ is $(0,\delta,1,y)$-homogeneous, where $y\in B_{1}(0^{n})\setminus B_{\tau}(V^{k})$,
\end{enumerate}
then $u$ is $(k+1,\epsilon,1,0^{n})$-homogeneous.
\end{lemma}

\begin{proof}
We argue by contradiction, assuming that there exists a sequence of $F$-subharmonic functions $u_{i}$ with $\|u\|_{L^{1}(B_{2}(0^{n}))}\leq\Lambda$ and satisfy the following properties:
\begin{enumerate}[~~~~~~(1)]
    \item $u_{i}$ is $(k,i^{-1},1,0^{n})$-homogeneous with respect to $k$-plane $V_{i}^{k}$;
    \item $u_{i}$ is $(0,i^{-1},1,y_{i})$-homogeneous, where $y_{i}\in B_{1}(0^{n})\setminus B_{\tau}(V_{i}^{k})$;
    \item $u_{i}$ is not $(k+1,\epsilon,1,0^{n})$-homogeneous.
\end{enumerate}
After passing to a subsequence, we assume that $\lim_{i\rightarrow\infty}V_{i}^{k}=V^{k}$, $\lim_{i\rightarrow\infty}y_{i}=y\in\overline{B_{1}(0^{n})}\setminus B_{2\tau}(V^{k})$ and $u_{i}$ converge to $u$ in $L_{loc}^{1}(B_{2}(0^{n}))$, where $u$ is a $F$-subharmonic function (see Lemma \ref{F subharmonic converge lemma}). By (1), (2) and Lemma \ref{homogeneous converge lemma}, there exists a function $h$ such that
\begin{enumerate}[~~~~~~(a)]
    \item $h$ is $k$-homogeneous at $0^{n}$ with respect to $V^{k}$;
    \item $h$ is $0$-homogeneous at $y$;
    \item $h=u$ in $B_{2}(0^{n})$.
\end{enumerate}
Hence, by Theorem \ref{cone-splitting principle}, we obtain that $h$ is a $(k+1)$-homogeneous function. Combining this with $u_{i}$ converge to $u$ in $L_{loc}^{1}(B_{2}(0^{n}))$ and (c), it is clear that $u_{i}$ is $(k+1,\epsilon,1,0^{n})$-homogeneous when $i$ is sufficiently large, which is a contradiction.
\end{proof}

\section{Top stratification of $\mathcal{S}(u)$}
In this section, we give proofs of (1) and (2) in Theorem \ref{first main result}.

\begin{proof}[Proof of (1) in Theorem \ref{first main result}]
For any $r\in(0,\frac{1}{5})$, $\{B_{r}(x)\}_{x\in E_{\eta}(u)\cap B_{1}(0^{n})}$ is a covering of $B_{r}(E_{\eta}(u)\cap B_{1}(0^{n}))$, where $E_{\eta}(u)=\{x\in B_{2}(0^{n})~|~\Theta^{S}(u,x)\geq \eta\}$. We take a Vitali covering $\{B_{r}(x_{i})\}_{i=1}^{M}$ such that
\begin{enumerate}[~~~~~~(a)]
    \item $B_{r}(x_{i})\cap B_{r}(x_{j})=\emptyset$ for any $i\neq j$;
    \item $B_{r}(E_{\eta}(u)\cap B_{1}(0^{n}))\subset\bigcup_{i}B_{5r}(x_{i})$;
    \item $x_{i}\in E_{\eta}(u)\cap B_{1}(0^{n})$ for each $i$.
\end{enumerate}
For each $x_{i}$, by the properties of $S(u,x_{i},\cdot)$ (see \cite[Corollary 5.3, Theorem 6.4]{HL1}), we have
\begin{equation*}
\lim_{t\rightarrow0}\frac{S_{-}'(u,x_{i},t)}{K'_{p}(t)}=\Theta^{S}(u,x_{i})
\end{equation*}
and
\begin{equation*}
\frac{S_{-}'(u,x_{i},t)}{K'_{p}(t)} \textit{~~is nondecreasing with respect to $t$}.
\end{equation*}
Since $x_{i}\in E_{\eta}(u)\cap B_{1}(0^{n})$, it then follows that
\begin{equation*}
\frac{S_{-}'(u,x_{i},r)}{K'_{p}(r)}\geq\Theta^{S}(u,x_{i})\geq\eta.
\end{equation*}
Using $S_{-}'(u,x_{i},r)=C(n)K_{n}'(r)\int_{B_{r}(x_{i})}\Delta u$ (see e.g. \cite[Theorem 3.2.16]{Ho} or \cite[p.33]{HL1}), it is clear that
\begin{equation*}
\int_{B_{r}(x_{i})}\Delta u\geq C(p,n)\eta r^{n-p}.
\end{equation*}
By (a), we obtain
\begin{equation}\label{top stratum proof equation 1}
\int_{B_{1+r}(0^{n})}\Delta u\geq\sum_{i=1}^{M}\int_{B_{r}(x_{i})}\Delta u\geq C(p,n)\eta Mr^{n-p}.
\end{equation}
Combining (b) and (\ref{top stratum proof equation 1}), we get
\begin{equation}\label{top stratum proof equation 2}
\text{Vol}(B_{r}(E_{\eta}(u)\cap B_{1}(0^{n})))\leq\sum_{i=1}^{M}\text{Vol}(B_{5r}(x_{i}))\leq C(p,n)\eta^{-1}\left(\int_{B_{1+r}(0^{n})}\Delta u\right)r^{p}.
\end{equation}
On the other hand, for every $y\in\mathcal{S}_{\eta}^{n-p}(u)\cap B_{1}(0^{n})$, since $0$ is a $(n-p+1)$-homogeneous function, by the definition of $\mathcal{S}_{\eta}^{n-p}(u)$, we have
\begin{equation}\label{top stratum proof equation 3}
\|u_{y,r}-0\|_{L^{1}(B_{1}(0^{n}))}\geq\eta,
\end{equation}
for any $r\in(0,1)$. Now, we take $U\in T_{y}(u)$. Combining (\ref{top stratum proof equation 3}) and the definition of tangent, it is clear that $\|U\|_{L^{1}(B_{1}(0^{n}))}\geq\eta$. By \cite[Theorem 10.1]{HL1}, we have
\begin{equation*}
\tilde{C}(p,n)\Theta^{S}(u,y)\geq-\int_{B_{1}(0^{n})}U=\|U\|_{L^{1}(B_{1}(0^{n}))}\geq\eta,
\end{equation*}
which implies $\mathcal{S}_{\eta}^{n-p}(u)\cap B_{1}(0^{n})\subset E_{\tilde{C}^{-1}\eta}(u)\cap B_{1}(0^{n})$. Then by (\ref{top stratum proof equation 2}) (replace $\eta$ by $\tilde{C}^{-1}\eta$), we obtain
\begin{equation*}
\text{Vol}(B_{r}(\mathcal{S}_{\eta}^{n-p}(u)\cap B_{1}(0^{n})))\leq C(p,n)\tilde{C}(p,n)\eta^{-1}\left(\int_{B_{1+r}(0^{n})}\Delta u\right)r^{p},
\end{equation*}
as required.
\end{proof}

\begin{proof}[Proof of (2) in Theorem \ref{first main result}]
We argue by contradiction, assuming that there exists a point $x\in\mathcal{S}(u)\setminus\mathcal{S}^{n-p}(u)$. By definition, there exists $\varphi\in T_{x}(u)$ such that $\varphi$ is $(n-p+1)$-homogeneous but not $n$-homogeneous. It is clear that
\begin{equation*}
dim_{H}(\mathcal{S}(\varphi))\geq n-p+1,
\end{equation*}
where $dim_{H}(\mathcal{S}(\varphi))$ denotes the Hausdorff dimension of $\mathcal{S}(\varphi)$. By (\ref{top stratum proof equation 2}) (replace $u$ by $\varphi$), we get $dim_{H}(E_{\eta}(\varphi)\cap B_{1}(0^{n}))\leq n-p$. By the similar argument, it is clear that $dim_{H}(E_{\eta}(\varphi))\leq n-p$. Since $\mathcal{S}(\varphi)=\bigcup_{\eta}E_{\eta}(\varphi)$, we get
\begin{equation*}
dim_{H}(\mathcal{S}(\varphi))\leq n-p,
\end{equation*}
which is a contradiction.
\end{proof}

\section{Hausdorff dimension of $\mathcal{S}^{k}(u)$}
In this section, we study the Hausdorff dimension of $\mathcal{S}^{k}(u)$. We use an iterated blow up argument as in \cite{ChCo} to prove (3) of Theorem \ref{first main result}. For convenience, we use $T_{x}(u)$ to denote the tangent set to $u$ at $x$ in the following argument.

\begin{lemma}\label{Hausdorff dimension lemma 1}
Let $h$ be a $F$-subharmonic function which is $k$-homogeneous at $0^{n}$ with respect to $k$-plane $V^{k}$. For any $x_{0}\notin V^{k}$, if $\varphi\in T_{x_{0}}(h)$, then $\varphi$ is $(k+1)$-homogeneous at $0^{n}$ with respect to $(k+1)$-plane $V^{k+1}=span\{x_{0},V^{k}\}$.
\end{lemma}

\begin{proof}
By the definition of tangent, there exists a sequence $\{r_{i}\}$ ($\lim_{i\rightarrow\infty}r_{i}=0$) such that $h_{x_{0},r_{i}}$ converge to $\varphi$ in $L_{loc}^{1}(\mathbf{R}^{n})$. Since $\varphi$ is subharmonic, in order to prove Lemma \ref{Hausdorff dimension lemma 1}, it suffices to prove
\begin{equation}\label{Hausdorff dimension lemma 1 equation 1}
\int_{B_{r}(y)}\varphi(x)dx=\int_{B_{r}(y+v)}\varphi(x)dx,
\end{equation}
for any $y\in\mathbf{R}^{n}$, $v\in V^{k+1}$ and $r>0$. First, we consider the case $p>2$.

\bigskip
\noindent
{\bf Case 1.} $p>2$.

\bigskip
We split up into different subcases.

\bigskip
\noindent
{\bf Subcase 1.1.} $v=\lambda x_{0}$ for some $\lambda\in\mathbf{R}$.

\bigskip
By direct calculations, we have
\begin{equation}\label{Hausdorff dimension lemma 1 equation 2}
\begin{split}
\int_{B_{r}(y+v)}\varphi(x)dx & = \int_{B_{r}(y+\lambda x_{0})}\varphi(x)dx\\
& = \lim_{i\rightarrow\infty}\int_{B_{r}(y+\lambda x_{0})}h_{x_{0},r_{i}}(x)dx\\
& = \lim_{i\rightarrow\infty}\int_{B_{r}(y+\lambda x_{0})}r_{i}^{p-2}h(x_{0}+r_{i}x)dx\\
& = \lim_{i\rightarrow\infty}\int_{B_{r}(0^{n})}r_{i}^{p-2}h(x_{0}+r_{i}x+r_{i}y+\lambda r_{i}x_{0})dx.
\end{split}
\end{equation}
Since $h$ is homogeneous, it is clear that
\begin{equation}\label{Hausdorff dimension lemma 1 equation 3}
\begin{split}
& \int_{B_{r}(0^{n})}r_{i}^{p-2}h(x_{0}+r_{i}x+r_{i}y+\lambda r_{i}x_{0})dx\\
= & \int_{B_{r}(0^{n})}(1+\lambda r_{i})^{2-p}r_{i}^{p-2}h(x_{0}+\frac{r_{i}x}{1+\lambda r_{i}}+\frac{r_{i}y}{1+\lambda r_{i}})dx\\
= & \int_{B_{\frac{r}{1+\lambda r_{i}}}(\frac{y}{1+\lambda r_{i}})}(1+\lambda r_{i})^{n+2-p}h_{x_{0},r_{i}}(x)dx.
\end{split}
\end{equation}
On the other hand, since $h_{x_{0},r_{i}}$ converge to $\varphi$ in $L_{loc}^{1}(\mathbf{R}^{n})$, it then follows that
\begin{equation}\label{Hausdorff dimension lemma 1 equation 4}
\begin{split}
& \lim_{i\rightarrow\infty}\int_{B_{\frac{r}{1+\lambda r_{i}}}(\frac{y}{1+\lambda r_{i}})}(1+\lambda r_{i})^{n+2-p}|h_{x_{0},r_{i}}(x)-\varphi(x)|dx\\
\leq & \lim_{i\rightarrow\infty}\int_{B_{r+1}(y)}2|h_{x_{0},r_{i}}(x)-\varphi(x)|dx\\
= & ~ 0.
\end{split}
\end{equation}
Combining (\ref{Hausdorff dimension lemma 1 equation 2}), (\ref{Hausdorff dimension lemma 1 equation 3}) and (\ref{Hausdorff dimension lemma 1 equation 4}), we obtain
\begin{equation*}
\begin{split}
& |\int_{B_{r}(y+v)}\varphi(x)dx-\int_{B_{r}(y)}\varphi(x)dx|\\
= & ~ |\lim_{i\rightarrow\infty}\int_{B_{\frac{r}{1+\lambda r_{i}}}(\frac{y}{1+\lambda r_{i}})}(1+\lambda r_{i})^{n+2-p}h_{x_{0},r_{i}}(x)dx-\int_{B_{r}(y)}\varphi(x)dx|\\
\leq & \lim_{i\rightarrow\infty}\int_{B_{\frac{r}{1+\lambda r_{i}}}(\frac{y}{1+\lambda r_{i}})}(1+\lambda r_{i})^{n+2-p}|h_{x_{0},r_{i}}(x)-\varphi(x)|dx\\
& + |\lim_{i\rightarrow\infty}\int_{B_{\frac{r}{1+\lambda r_{i}}}(\frac{y}{1+\lambda r_{i}})}(1+\lambda r_{i})^{n+2-p}\varphi(x)dx-\int_{B_{r}(y)}\varphi(x)dx|\\
\leq & ~ 0,
\end{split}
\end{equation*}
where we used Lebesgue's dominated convergence theorem for the last inequality. This completes the proof of Subcase 1.1.

\bigskip
\noindent
{\bf Subcase 1.2.} $v\in V^{k}$.

\bigskip
By similar calculations in Subcase 1.1, we have
\begin{equation}\label{Hausdorff dimension lemma 1 equation 5}
\int_{B_{r}(y+v)}\varphi(x)dx=\lim_{i\rightarrow\infty}\int_{B_{r}(0^{n})}r_{i}^{p-2}h(x_{0}+r_{i}x+r_{i}y+r_{i}v)dx.
\end{equation}
Since $h$ is $k$-homogeneous with respect to $k$-plane $V^{k}$, it is clear that
\begin{equation}\label{Hausdorff dimension lemma 1 equation 6}
\begin{split}
\int_{B_{r}(0^{n})}r_{i}^{p-2}h(x_{0}+r_{i}x+r_{i}y+r_{i}v)dx
& = \int_{B_{r}(0^{n})}r_{i}^{p-2}h(x_{0}+r_{i}x+r_{i}y)dx\\
& = \int_{B_{r}(y)}h_{x_{0},r_{i}}(x)dx.
\end{split}
\end{equation}
Combining (\ref{Hausdorff dimension lemma 1 equation 5}), (\ref{Hausdorff dimension lemma 1 equation 6}) and $h_{x_{0},r_{i}}$ converge to $\varphi$ in $L_{loc}^{1}(\mathbf{\mathbf{R}}^{n})$, we get (\ref{Hausdorff dimension lemma 1 equation 1}), which completes the proof of Subcase 1.2.

Next, we consider the case $p=2$.

\bigskip
\noindent
{\bf Case 2.} $p=2$.

\bigskip
Similarly, we split up into different subcases.

\bigskip
\noindent
{\bf Subcase 2.1.} $v=\lambda x_{0}$ for some $\lambda\in\mathbf{R}$.

\bigskip
By the definition of tangential 2-flow (see Definition \ref{definition of tangential p-flow}), we have
\begin{equation*}
\begin{split}
\int_{B_{r}(y+v)}\varphi(x)dx & = \int_{B_{r}(y+\lambda x_{0})}\varphi(x)dx\\
& = \lim_{i\rightarrow\infty}\int_{B_{r}(y+\lambda x_{0})}h_{x_{0},r_{i}}(x)dx\\
& = \lim_{i\rightarrow\infty}\int_{B_{r}(y+\lambda x_{0})}\left(h(x_{0}+r_{i}x)-M(h,x_{0},r_{i})\right)dx\\
& = \lim_{i\rightarrow\infty}\int_{B_{r}(0^{n})}\left(h(x_{0}+r_{i}x+r_{i}y+\lambda r_{i}x_{0})-M(h,x_{0},r_{i})\right)dx.
\end{split}
\end{equation*}
By the homogeneity of $h$, we obtain
\begin{equation*}
\begin{split}
  & \int_{B_{r}(0^{n})}\left(h(x_{0}+r_{i}x+r_{i}y+\lambda r_{i}x_{0})-M(h,x_{0},r_{i})\right)dx\\
= & \int_{B_{r}(0^{n})}\left( h(x_{0}+\frac{r_{i}x}{1+\lambda r_{i}}+\frac{r_{i}y}{1+\lambda r_{i}})+M(h,0^{n},1+\lambda r_{i})-M(h,x_{0},r_{i})\right)dx\\
= & \int_{B_{\frac{r}{1+\lambda r_{i}}}(\frac{y}{1+\lambda r_{i}})}h_{x_{0},r_{i}}(x)dx+\int_{B_{r}(0^{n})}M(h,0^{n},1+\lambda r_{i})dx.
\end{split}
\end{equation*}
Since $h$ is homogeneous, we get $M(h,0^{n},1)=0$. By the continuity of $M(h,0^{n},\cdot)$, it is clear that
\begin{equation*}
\int_{B_{r}(y+v)}\varphi(x)dx=\lim_{i\rightarrow\infty}\int_{B_{\frac{r}{1+\lambda r_{i}}}(\frac{y}{1+\lambda r_{i}})}h_{x_{0},r_{i}}(x)dx.
\end{equation*}
By the similar argument in Subcase 1.1, we complete the proof of Subcase 2.1.

\bigskip
\noindent
{\bf Subcase 2.2.} $v\in V^{k}$.
The proof of Subcase 2.2 is similar to the proof of Subcase 1.2.

\bigskip
\end{proof}

\begin{lemma}\label{Hausdorff dimension lemma 2}
Let $u$ be a $F$-subharmonic function on $B_{2}(0^{n})$. If $\text{Haus}^{l}(\mathcal{S}^{k}(u))>0$ for $l>k$, then $\text{Haus}^{l}(A)>0$, where
\begin{equation*}
A:=\{y\in B_{2}(0^{n})~|~\text{there exists a tangent $\varphi\in T_{y}(u)$ such that $\text{Haus}^{l}(\mathcal{S}^{k}(\varphi))>0$}\}.
\end{equation*}
\end{lemma}

\begin{proof}
Combining $\text{Haus}^{l}(\mathcal{S}^{k}(u))>0$ and $\mathcal{S}^{k}(u)=\bigcup_{\eta}\mathcal{S}_{\eta}^{k}(u)$ (see (\ref{quantitative stratification})), there exists a constant $\eta_{0}>0$ such that $\text{Haus}^{l}(\mathcal{S}_{\eta_{0}}^{k}(u))>0$. By the property of Hausdorff measure, we have $\text{Haus}^{l}(\mathcal{S}_{\eta_{0}}^{k}(u)\setminus D_{\eta_{0}}^{l}(u))=0$, where
\begin{equation*}
D_{\eta_{0}}^{l}(u)=\{x\in \mathcal{S}_{\eta_{0}}^{k}(u)~|~\limsup_{r\rightarrow0}\frac{\text{Haus}_{\infty}^{l}(\mathcal{S}_{\eta_{0}}^{k}(u)\cap B_{r}(x))}{\omega_{l}r^{l}}\geq 2^{-l}\}.
\end{equation*}
Therefore, in order to prove Lemma \ref{Hausdorff dimension lemma 2}, it suffices to prove that there exists a tangent $\varphi\in T_{y}(u)$ such that $\text{Haus}^{l}(\mathcal{S}^{k}(\varphi))>0$ for any $y\in D_{\eta_{0}}^{l}(u)$. By the definition of $D_{\eta_{0}}^{l}(u)$, there exists a sequence of $\{r_{j}\}$ ($\lim_{i\rightarrow\infty}r_{j}=0$) such that
\begin{equation*}
\lim_{j\rightarrow\infty}\frac{\text{Haus}_{\infty}^{l}(\mathcal{S}_{\eta_{0}}^{k}(u)\cap B_{r_{j}}(y))}{\omega_{l}r_{j}^{l}}\geq 2^{-l}.
\end{equation*}
If $y+r_{i}z\in\mathcal{S}_{\eta_{0}}^{k}(u)\cap B_{r_{j}}(y)$, then $z\in\mathcal{S}_{\eta_{0}}^{k}(u_{y,r_{j}})\cap B_{1}(0^{n})$. Combining this and the definition of Hausdorff measure, we have
\begin{equation*}
\lim_{j\rightarrow\infty}\text{Haus}_{\infty}^{l}(\mathcal{S}_{\eta_{0}}^{k}(u_{y,r_{j}})\cap B_{1}(0^{n}))\geq 2^{-l}.
\end{equation*}
After passing to a subsequence, we can assume that $u_{y,r_{j}}$ converge to $\varphi\in T_{y}(u)$ in $L_{loc}^{1}(\mathbf{R}^{n})$.

\bigskip
\noindent
{\bf Claim.}  If $z_{j}\in \mathcal{S}_{\eta_{0}}^{k}(u_{y,r_{j}})$ and $\lim_{j\rightarrow\infty}z_{j}=z$, then $z\in \mathcal{S}_{\eta_{0}}^{k}(\varphi)$.

\bigskip
\begin{proof}[Proof of Claim]
For any $r\in(0,1)$ and $(k+1)$-homogeneous function $h$, we have
\begin{equation*}
\begin{split}
& \int_{B_{1}(0^{n})}|\varphi_{z,r}(x)-h(x)|dx\\
\geq & \int_{B_{1}(0^{n})}|(u_{y,r_{j}})_{z_{j},r}(x)-h(x)|dx
-\int_{B_{1}(0^{n})}|\varphi_{z_{j},r}(x)-(u_{y,r_{j}})_{z_{j},r}(x)|dx-\int_{B_{1}(0^{n})}|\varphi_{z,r}(x)-\varphi_{z_{j},r}(x)|dx.
\end{split}
\end{equation*}
Letting $j\rightarrow\infty$, by Lemma \ref{tangential p-flow stable lemma}, we obtain
\begin{equation*}
\int_{B_{1}(0^{n})}|\varphi_{z,r}(x)-h(x)|dx\geq\eta_{0},
\end{equation*}
which implies $z\in \mathcal{S}_{\eta_{0}}^{k}(\varphi)$. We complete the proof of Claim.
\end{proof}

Combining Claim and the property of Hausdorff measure, it is clear that
\begin{equation*}
\text{Haus}^{l}(\mathcal{S}_{\eta_{0}}^{k}(\varphi)\cap B_{1}(0^{n}))\geq\lim_{j\rightarrow\infty}\text{Haus}_{\infty}^{l}(\mathcal{S}_{\eta_{0}}^{k}(u_{y,r_{j}})\cap B_{1}(0^{n}))\geq 2^{-l}>0,
\end{equation*}
as desired.
\end{proof}

\begin{theorem}
Let $u$ be a $F$-subharmonic function on $B_{2}(0^{n})$. Then for any $1\leq k\leq n$, we have
\begin{equation*}
dim_{H}(\mathcal{S}^{k}(u))\leq k.
\end{equation*}
\end{theorem}

\begin{proof}
We argue by contradiction. Suppose that $\text{Haus}^{l}(\mathcal{S}^{k}(u))>0$ for some $l>k$. By Lemma \ref{Hausdorff dimension lemma 2}, there exists $y_{0}\in \mathcal{S}^{k}(u)$ and $\varphi_{0}\in T_{y_{0}}(u)$ such that $\text{Haus}^{l}(\mathcal{S}^{k}(\varphi_{0}))>0$. We assume that $\varphi_{0}$ is $m$-homogeneous with respect to $m$-plane $V_{0}^{m}$, where $m\leq k$. By Lemma \ref{Hausdorff dimension lemma 2}, $\text{Haus}^{l}(\mathcal{S}^{k}(\varphi_{0}))>0$ and $m\leq k<l$, there exists $y_{1}\in \mathcal{S}^{k}(\varphi_{0})\setminus V_{0}^{m}$ and $\varphi_{1}\in T_{y_{1}}(\varphi_{0})$ such that $\text{Haus}^{l}(\mathcal{S}^{k}(\varphi_{1}))>0$. By Lemma \ref{Hausdorff dimension lemma 1}, we obtain that $\varphi_{1}$ is $(m+1)$-homogeneous with respect to $(m+1)$-plane $V_{1}^{m+1}=span\{V_{0}^{m},y_{1}\}$. Repeating this process, there exist $y_{k-m+1}\in \mathcal{S}^{k}(\varphi_{k-m})\setminus V_{k-m}^{k}$ and $\varphi_{k-m+1}\in T_{y_{k-m+1}}(\varphi_{k-m})$ such that $\varphi_{k-m+1}$ is $(k+1)$-homogeneous, which contradicts with the definition of $\mathcal{S}^{k}(\varphi_{k-m})$.
\end{proof}

\section{Rectifiability of $\mathcal{S}^{k}(u)$}
In this section, we prove the $k^{th}$ stratification $\mathcal{S}^{k}(u)$ is $k$-rectifiable when uniqueness of tangents holds for $F$ (i.e., Theorem \ref{Rectifiability theorem}). Let $u$ be a $F$-subharmonic function on $B_{2}(0^{n})$ with $\|u\|_{L^{1}(B_{1}(0^{n}))}\leq\Lambda$. First, we define
\begin{equation*}
F_{\delta,\eta}(u)=\{x\in B_{2}(0^{n})~|~\text{$u$ is $(0,\delta,r,x)$-homogeneous for any $r\in(0,\eta)$}\}.
\end{equation*}
For any $x\in \left(F_{\delta,\eta}(u)\cap \mathcal{S}_{\epsilon}^{k}(u)\right)\setminus\mathcal{S}^{k-1}(u)$, where $\epsilon>0$, let $\varphi$ be the unique tangent to $u$ at $x$. We assume $\varphi$ is $k$-homogeneous with respect to $k$-plane $V_{\varphi}^{k}$. It then follows that $\|\varphi\|_{L^{1}(B_{2}(0^{n}))}\leq\Lambda_{1}(\Lambda,F)$.

\begin{lemma}\label{Rectifiability lemma 1}
For any $\tau\in(0,1)$, there exists $r_{x}$ such that for any $r<r_{x}$, we have
\begin{equation*}
F_{\delta,1}(u_{x,r})\subset B_{2\tau}(V_{\varphi}^{k}),
\end{equation*}
where $\delta=\delta(\epsilon,2\tau,\Lambda_{1},F)$ is the constant in Lemma \ref{cone-splitting lemma}.
\end{lemma}

\begin{proof}
We argue by contradiction, assuming that there exist $\{r_{i}\}$ and $\{z_{i}\}$ such that $\lim_{i\rightarrow\infty}r_{i}=0$ and $z_{i}\in F_{\delta,1}(u_{x,r_{i}})\setminus B_{2\tau}(V_{\varphi}^{k})$. It then follows that there exists homogeneous function $h_{i}$ such that
\begin{equation*}
\int_{B_{1}(0^{n})}|(u_{x,r_{i}})_{z_{i},r}(y)-h_{i}(y)|dy\leq\delta
\end{equation*}
for any $r\in(0,1)$. Since $u_{x,r_{i}}$ converge to $\varphi$ in $L_{loc}^{1}(\mathbf{R}^{n})$, by Lemma \ref{tangential p-flow stable lemma}, after passing to a subsequence, we can assume that $\lim_{i\rightarrow\infty}z_{i}=z$ and $h_{i}$ converge to $h$ in $L_{loc}^{1}(B_{2}(0^{n}))$. For any $r\in(0,1)$, by Lemma \ref{tangential p-flow stable lemma}, we have
\begin{equation*}
\begin{split}
& \int_{B_{1}(0^{n})}|\varphi_{z,r}(y)-h(y)|dy\\
\leq & \lim_{i\rightarrow\infty}\int_{B_{1}(0^{n})}|\varphi_{z,r}(y)-(u_{x,r_{i}})_{z_{i},r}(y)|dy+\lim_{i\rightarrow\infty}\int_{B_{1}(0^{n})}|(u_{x,r_{i}})_{z_{i},r}(y)-h_{i}(y)|dy\\
&+\lim_{i\rightarrow\infty}\int_{B_{1}(0^{n})}|h_{i}(y)-h(y)|dy\\
\leq & ~ \delta
\end{split}
\end{equation*}
which implies $z\in F_{\delta,1}(\varphi)\setminus B_{2\tau}(V_{\varphi}^{k})$. However, by Lemma \ref{cone-splitting lemma} and $x\in S_{\epsilon}^{k}(u)$, we get $F_{\delta,1}(\varphi)\subset B_{\tau}(V_{\varphi}^{k})$, which is a contradiction.
\end{proof}

\begin{lemma}\label{Rectifiability lemma 2}
For any $r\leq r_{x}$, we have
\begin{equation*}
F_{\delta,r}(u)\cap B_{r}(x)\subset B_{2\tau r}(V_{\varphi}^{k}+x)
\end{equation*}
\end{lemma}

\begin{proof}
For any $x+rz\in F_{\delta,r}(u)\cap B_{r}(x)$, where $z\in B_{1}(0^{n})$, there exists homogeneous function $h$ such that for any $s\in(0,r)$, we have
\begin{equation*}
\int_{B_{1}(0^{n})}|u_{x+rz,s}(y)-h(y)|dy\leq\delta
\end{equation*}
It then follows that
\begin{equation*}
\int_{B_{1}(0^{n})}|(u_{x,r})_{z,\frac{s}{r}}(y)-h(y)|dy\leq\delta,
\end{equation*}
which implies $z\in F_{\delta,1}(u_{x,r})$. Combining this with Lemma \ref{Rectifiability lemma 1}, we have $x+rz\in B_{2\tau r}(V_{\varphi}^{k}+x)$.
\end{proof}

Now, we are in a position to prove Theorem \ref{Rectifiability theorem}.

\begin{proof}[Proof of Theorem \ref{Rectifiability theorem}]
For any $\eta>0$ and $x\in\left(F_{\delta,\eta}(u)\cap \mathcal{S}_{\epsilon}^{k}(u)\right)\setminus\mathcal{S}^{k-1}(u)$, by Lemma \ref{Rectifiability lemma 2}, there exists $r_{x}\leq\eta$ such that for any $r<r_{x}$, we have $F_{\delta,r}(u)\cap B_{r}(x)\subset B_{2\tau r}(V_{\varphi}^{k}+x)$, which implies
\begin{equation*}
\left(\left(F_{\delta,\eta}(u)\cap \mathcal{S}_{\epsilon}^{k}(u)\right)\setminus\mathcal{S}^{k-1}(u)\right)\cap B_{r}(x)\subset B_{2\tau r}(V_{\varphi}^{k}+x).
\end{equation*}
Hence, $\left(F_{\delta,\eta}(u)\cap \mathcal{S}_{\epsilon}^{k}(u)\right)\setminus\mathcal{S}^{k-1}(u)$ is $k$-rectifiable (see e.g. \cite[p.61, Lemma 1]{Si}). Since uniqueness of tangents holds for $F$, we have $\mathcal{S}_{\epsilon}^{k}(u)=\cup_{\eta}(F_{\delta,\eta}(u)\cap \mathcal{S}_{\epsilon}^{k}(u))$. By (\ref{quantitative stratification}), it then follows that
\begin{equation*}
\begin{split}
\mathcal{S}^{k}(u)\setminus\mathcal{S}^{k-1}(u)
& =  \bigcup_{\epsilon}\left(\mathcal{S}_{\epsilon}^{k}(u)\setminus\mathcal{S}^{k-1}(u)\right)\\
& =  \bigcup_{\epsilon}\bigcup_{\eta}\left(\left(F_{\delta,\eta}(u)\cap \mathcal{S}_{\epsilon}^{k}(u)\right)\setminus\mathcal{S}^{k-1}(u)\right),
\end{split}
\end{equation*}
which implies $\mathcal{S}^{k}(u)\setminus\mathcal{S}^{k-1}(u)$ is $k$-rectifiable. On the other hand, since uniqueness of tangents holds for $F$ implies homogeneity of tangents holds for $F$, by (3) of Theorem \ref{first main result}, we have $\text{Haus}^{k}(\mathcal{S}^{k-1}(u))=0$. It then follows that $\mathcal{S}^{k-1}(u)$ is $k$-rectifiable. Hence, $\mathcal{S}^{k}(u)=\left(\mathcal{S}^{k}(u)\setminus\mathcal{S}^{k-1}(u)\right)\cup\mathcal{S}^{k-1}(u)$ is $k$-rectifiable.
\end{proof}

\section{$F$-subharmonic functions}
In this section, we consider the singular sets of $F$-subharmonic functions and give the proof of Theorem \ref{estimate of singular set}. We assume that strong uniqueness of tangents holds for $F$ and $p>2$, where $p$ is the Riesz characteristic of $F$. By \cite[Proposition 7.1, (12.3)]{HL1}, all density functions are equivalent, i.e., $\Theta^{M}=\Theta^{S}=\frac{n-p+2}{n}\Theta^{V}$. For convenience, if $u$ is a $F$-subharmonic function on $B_{2}(0^{n})$, we use $E_{c}(u)$ to denote the set $\{x\in B_{2}(0^{n})~|~\Theta^{V}(u,x)\geq c\}$ in this section.

\subsection{Monotonicity condition and $F$-energy}
In this subsection, we introduce the monotonicity condition and $F$-energies of $F$-subharmonic functions. And then we prove every $F$-subharmonic function satisfies monotonicity condition after subtracting a constant.
\begin{definition}
Let $u$ be a $F$-subharmonic function on $B_{2}(0^{n})$. We say that $u$ satisfies monotonicity condition if $F$-energy defined by
\begin{equation*}
\theta_{F}(u,x,r):=\frac{S(u,x,r)}{K_{p}(r)}+\frac{M(u,x,r)}{K_{p}(r)}
\end{equation*}
is nondecreasing in $r\in(0,\frac{1}{2})$ for any $x\in B_{1}(0^{n})$. And we define $\theta_F(u,x,0)=\lim_{r\rightarrow0}\theta_F(u,x,r)$
\end{definition}

\begin{lemma}\label{monotonicity condition lemma}
Let $u$ be a $F$-subharmonic function on $B_{2}(0^{n})$ with $\|u\|_{L^{1}(B_{2}(0^{n}))}\leq\Lambda$. Then there exists constant $N(\Lambda,p,n)$ such that $u-N$ satisfies monotonicity condition.
\end{lemma}

\begin{proof}
For any $x\in B_{1}(0^{n})$, since $S(u,x,\cdot)$ is $K_{p}$-convex (see \cite[p.31]{HL1}), by Lemma \ref{p-plane lemma}, we have
\begin{equation*}
\frac{S_{+}'(u,x,\frac{1}{2})}{K_{p}'(\frac{1}{2})}\leq\frac{S(u,x,\frac{2}{3})-S(u,x,\frac{1}{2})}{K_{p}(\frac{2}{3})-K_{p}(\frac{1}{2})}\leq\tilde{N}(\Lambda,p,n).
\end{equation*}
By the property of subharmonic functions, there exists constant $N(\Lambda,p,N)$ such that
\begin{equation*}
S(u,x,\frac{1}{2})-\tilde{N}(\Lambda,p,n)K_{p}(\frac{1}{2})\leq N(\Lambda,p,n).
\end{equation*}
Hence, by the property of $K_{p}$-convex function, we obtain
\begin{equation*}
\frac{S(u-N,x,r)}{K_{p}(r)}=\frac{S(u,x,r)-N}{K_{p}(r)-0}
\end{equation*}
is nondecreasing in $r\in(0,\frac{1}{2})$. Similarly, by increasing the value of $N$, we can prove $\frac{M(u-N,x,r)}{K_{p}(r)}$ is also nondecreasing. And this completes the proof.
\end{proof}

\subsection{Quantitative rigidity results}
In this subsection, we prove some quantitative rigidity results of $F$-subharmonic functions.
\begin{lemma}\label{singularity preserve lemma}
Let $u_{i}$ and $u$ be $F$-subharmonic functions on $B_{2}(0^{n})$. For $c>0$, if $u_{i}$ converge to $u$ in $L_{loc}^{1}(B_{2}(0^{n}))$ and $x_{i}$ converge to $x$, where $x_{i}\in E_{c}(u_{i})\cap\overline{B_{1}(0^{n})}$, then
\begin{equation*}
x\in E_{c}(u)\cap\overline{B_{1}(0^{n})}.
\end{equation*}
\end{lemma}

\begin{proof}
For any $t>0$, we compute
\begin{equation*}
\begin{split}
&~~~~~~|V(u,x,t)-V(u_{i},x_{i},t)|\\
& \leq \frac{1}{\omega_{n}t^{n}}\int_{B_{t}(x_{i})}|u(y)-u_{i}(y)|dy+\frac{1}{\omega_{n}t^{n}}|\int_{B_{t}(x)}u(y)dy-\int_{B_{t}(x_{i})}u(y)dy|\\
& \leq\frac{1}{\omega_{n}t^{n}}\int_{B_{1+t}(0^{n})}|u(y)-u_{i}(y)|dy+\frac{1}{\omega_{n}t^{n}}|\int_{B_{t}(x)}u(y)dy-\int_{B_{t}(x_{i})}u(y)dy|.
\end{split}
\end{equation*}
which implies
\begin{equation*}
\lim_{i\rightarrow\infty}V(u_{i},x_{i},t)=V(u,x,t).
\end{equation*}
Therefore, for any $0<s<r<\frac{1}{2}$, we obtain
\begin{equation*}
\frac{V(u,x,r)-V(u,x,s)}{K_{p}(r)-K_{p}(s)}=\lim_{i\rightarrow\infty}\frac{V(u_{i},x_{i},r)-V(u_{i},x_{i},s)}{K_{p}(r)-K_{p}(s)}\geq c,
\end{equation*}
where we used the condition $x_{i}\in E_{c}(u_{i})\cap\overline{B_{1}(0^{n})}$. By the definition of density function $\Theta^{V}$ (see Corollary 5.3 in \cite{HL1}), we obtain $\Theta(u,x)\geq c$. This completes the proof.
\end{proof}

\begin{lemma}\label{rigidity for F-subharmonic}
Let $u$ be a $F$-subharmonic function on $B_{2}(0^{n})$ with $\|u\|_{L^{1}(B_{2}(0^{n}))}\leq\Lambda$ and satisfies monotonicity condition. For any $\epsilon>0$, there exists constant $\delta_{0}(\epsilon,c,\Lambda,F)$ such that if
\begin{equation*}
\theta_{F}(u,0^{n},\frac{1}{2})-\theta_{F}(u,0^{n},\delta_{0})<\delta_{0},
\end{equation*}
then $u$ is $(0,\epsilon,1,0^{n})$-homogeneous.
\end{lemma}

\begin{proof}
We argue by contradiction. Assuming that there exists a sequence of $F$-subharmonic function $u_{i}$ on $B_{2}(0^{n})$ such that
\begin{enumerate}[~~~~~~(1)]
    \item $\|u_{i}\|_{L^{1}(B_{2}(0^{n}))}\leq\Lambda$;
    \item $u_{i}$ satisfies monotonicity condition;
    \item $\theta_{F}(u_{i},0^{n},\frac{1}{2})-\theta_{F}(u_{i},0^{n},i^{-1})<i^{-1}$;
    \item $u_{i}$ is not $(0,\epsilon,1,0^{n})$-homogeneous.
\end{enumerate}
By Lemma \ref{F subharmonic converge lemma}, after passing to a subsequence, we can assume $u_{i}$ converge to $u$ in $L_{loc}^{1}(B_{2}(0^{n}))$, where $u$ is a $F$-subharmonic function. By Lemma \ref{M S stable lemma}, it is clear that $u$ also satisfies monotonicity condition. For each $t\in(0,\frac{1}{2})$, we obtain
\begin{equation*}
\begin{split}
\theta_{F}(u,0^{n},\frac{1}{2})-\theta_{F}(u,0^{n},t) &=\frac{S(u,0^{n},\frac{1}{2})}{K_{p}(\frac{1}{2})}-\frac{S(u,0^{n},t)}{K_{p}(t)}+\frac{M(u,0^{n},\frac{1}{2})}{K_{p}(\frac{1}{2})}-\frac{M(u,0^{n},t)}{K_{p}(t)}\\
&=\lim_{i\rightarrow\infty}\left(\theta_{F}(u_{i},0^{n},\frac{1}{2})-\theta_{F}(u_{i},0^{n},t)\right)\\
&\leq0,
\end{split}
\end{equation*}
which implies
\begin{equation*}
S(u,0^{n},r)=\Theta^{S}(u,0^{n})K_{p}(r)~~\text{and}~~M(u,0^{n},r)=\Theta^{M}(u,0^{n})K_{p}(r)
\end{equation*}
for any $r\in(0,\frac{1}{2})$, where $\Theta^{S}$ and $\Theta^{M}$ are $S$-density and $M$-density (see Section 6 of \cite{HL1}). Since strong uniqueness holds for $u$, then $\Theta^{S}=\Theta^{M}$ (see \cite[(12.3)]{HL1}). By the definitions of $S$ and $M$, we obtain
\begin{equation*}
u(x)=\Theta^{S}(u,0^{n})K_{p}(|x|).
\end{equation*}
Therefore, $u$ is $0$-homogenous. However, $u_{i}$ converge to $u$ in $L^{1}_{loc}(B_{2}(0^{n}))$. Thus, $u_{i}$ are $(0,\epsilon,1,0^{n})$-homogenous when $i$ is sufficiently large, which is a contradiction.
\end{proof}

\begin{lemma}\label{no singular point lemma}
Let $u$ be a $F$-subharmonic function on $B_{2}(0^{n})$ with $\|u\|_{L^{1}(B_{2}(0^{n}))}\leq\Lambda$. For any $c>0$, there exists constant $\epsilon(c,\Lambda,F)$ such that if $u$ is $(0,\epsilon,1,0^{n})$-homogenous, then
\begin{equation*}
E_{c}(u)\cap A_{\frac{1}{16},\frac{1}{2}}(0^{n})=\emptyset,
\end{equation*}
where $A_{\frac{1}{16},\frac{1}{2}}=\{x\in \mathbf{R}^{n}~|~\frac{1}{16}\leq|x|\leq\frac{1}{2}\}$.
\end{lemma}

\begin{proof}
We argue by contradiction, assuming that there exists a sequence of $F$-subharmonic functions $u_{i}$ on $B_{2}(0^{n})$ such that
\begin{enumerate}[~~~~~~(1)]
    \item $\|u_{i}\|_{L^{1}(B_{2}(0^{n}))}\leq\Lambda$;
    \item $u_{i}$ is $(0,i^{-1},1,0^{n})$-homogeneous;
    \item there exists point $x_{i}\in E_{c}(u_{i})\cap A_{\frac{1}{16},\frac{1}{2}}$.
\end{enumerate}
By Lemma \ref{F subharmonic converge lemma}), after passing to a subsequence, we can assume $u_{i}$ converge to $u$ in $L_{loc}^{1}(B_{2}(0^{n}))$ and $x_{i}$ converge to $x$, where $u$ is a $F$-subharmonic function. By (2), Lemma \ref{homogeneous converge lemma} and strong uniqueness holds for $F$, then there exists a constant $\Theta\geq0$ such that
\begin{equation}\label{no singular point lemma equation1}
u(x)=\Theta K_{p}(|x|) ~\text{~in $B_{1}(0^{n})$}.
\end{equation}
By (3) and Lemma \ref{singularity preserve lemma}, we have $x\in E_{c}(u)\cap A_{\frac{1}{16},\frac{1}{2}}$, which contradicts with (\ref{no singular point lemma equation1}).
\end{proof}

\begin{remark}
In \cite{HL1}, Harvey and Lawson proved the discreteness of $E_{c}(u)$ (see \cite[Theorem 14.1, Theorem 14.1']{HL1}). As an immediate corollary of Lemma \ref{no singular point lemma}, we also prove that every point in $E_{c}(u)$ is isolated, which gives another proof of discreteness of $E_{c}(u)$.
\end{remark}

\subsection{Proof of Theorem \ref{estimate of singular set}}
First, we have the following lemma.
\begin{lemma}\label{scale of F-subharmonic}
Let $u$ be a $F$-subharmonic function on $B_{2}(0^{n})$ with $\|u\|_{L^{1}(B_{2}(0^{n}))}\leq\Lambda$. For any $x\in B_{1}(0^{n})$, $r\in(0,\frac{1}{2})$, there exists constant $N(\Lambda,F)$ such that
\begin{equation*}
\int_{B_{1}(0^{n})}|u_{x,r}(y)|dy\leq N.
\end{equation*}
\end{lemma}

\begin{proof}
Without loss of generality, we assume $u\leq0$ on $B_{\frac{3}{2}}(0^{n})$. Since $V(u,x,\cdot)$ is $K_{p}$-convex, we have
\begin{equation*}
\frac{V(u,x,1)-V(u,x,r)}{K_{p}(1)-K_{p}(r)}\leq\frac{V(u,x,1)-V(u,x,\frac{1}{2})}{K_{p}(1)-K_{p}(\frac{1}{2})}\leq C(\Lambda,p,n),
\end{equation*}
which implies
\begin{equation*}
\frac{V(u,x,r)}{K_{p}(r)}\leq\frac{V(u,x,1)}{K_{p}(r)}+C(\Lambda,p,n)\frac{K_{p}(r)-K_{p}(1)}{K_{p}(r)} \leq N(\Lambda,n,p).
\end{equation*}
Since $u\leq0$ on $B_{\frac{3}{2}}(0^{n})$, it then follows that
\begin{equation*}
\int_{B_{1}(0^{n})}|u_{x,r}(y)|dy=-\int_{B_{1}(0^{n})}u_{x,r}(y)dy=\omega_{n}\frac{V(u,x,r)}{K_{p}(r)}\leq N(\Lambda,n,p),
\end{equation*}
as desired.
\end{proof}

Now, we are in the position to prove Theorem \ref{estimate of singular set}.
\begin{proof}[Proof of Theorem \ref{estimate of singular set}]
We split up into two cases.

\bigskip
\noindent
{\bf Case 1.} $u$ satisfies monotonicity condition.

\bigskip
For convenience, we use $S_{0}$ denote $\#\left(E_{c}(u)\cap B_{1}(0^{n})\right)$. And we will obtain an upper bound of $S_{0}$ by induction argument.

For $i=1$, we consider the covering $\{B_{2^{-1}}(x_{j})\}$ of $E_{c}(u)\cap B_{1}(0^{n})$ such that
\begin{enumerate}[~~~~~~(1)]
    \item $x_{j}\in E_{c}(u)\cap B_{1}(0^{n})$;
    \item $B_{2^{-2}}(x_{j})$ are disjoint.
\end{enumerate}
In this covering, there exists a ball containing the largest number of points in $E_{c}(u)\cap B_{1}(0^{n})$ (say $B_{2^{-1}}(x_{1})$, contains $S_{1}$ points in $E_{c}(u)\cap B_{1}(0^{n})$).

If $S_{1}=S_{0}$, we put $T_{1}=0$, otherwise put $T_{1}=1$. If $T_{1}=1$, by (2) and the definition of $S_{1}$, it is clear that
\begin{equation*}
2^{-2n}S_{0}\leq S_{1}<S_{0}.
\end{equation*}
Furthermore, in this case, we have
\begin{equation*}
\left(E_{c}(u)\cap B_{1}(0^{n})\right)\cap\left(B_{2}(z)\setminus B_{\frac{1}{4}}(z)\right)\neq\emptyset
\end{equation*}
for any $z\in B_{2^{-1}}(x_{1})$.

We repeat this process by covering $E_{c}(u)\cap B_{2^{-i}}(x_{i})$ with balls of radius $2^{-i-1}$. Since $E_{c}(u)\cap B_{1}(0^{n})$ is discrete, there exists $i_{0}\in\mathbb{Z}_{+}$ such that $S_{i_{0}}=1$. We define
\begin{equation*}
I:=\{1\leq i\leq i_{0}~|~T_{i}=1\}.
\end{equation*}
Then we obtain
\begin{equation}\label{estimate of singular set equation1}
S_{0}\leq(2^{2n})^{|I|}.
\end{equation}
In order to get an upper bound of $|I|$, we consider the point $x_{i_{0}}$. For any $i\in I$, by construction, we have
\begin{equation*}
\left(E_{c}(u)\cap B_{1}(0^{n})\right)\cap\left(B_{2^{-i+2}}(x_{i_{0}})\setminus B_{2^{-i-1}}(x_{i_{0}})\right)\neq\emptyset,
\end{equation*}
which implies
\begin{equation*}
E_{c}(u_{x_{i_{0}},2^{-i+3}})\cap A_{\frac{1}{16},\frac{1}{2}}\neq\emptyset.
\end{equation*}
Combining Lemma \ref{rigidity for F-subharmonic}, Lemma \ref{no singular point lemma} and Lemma \ref{scale of F-subharmonic}, it is clear that
\begin{equation*}
\theta(u_{x_{i_{0}},2^{-i+3}},0^{n},\frac{1}{2})-\theta(u_{x_{i_{0}},2^{-i+3}},0^{n},\delta_{0})\geq\delta_{0},
\end{equation*}
where $\delta_{0}(\epsilon,c,N,F)$, $\epsilon(c,N,F)$ and $N(\Lambda,F)$ are the constants in Lemma \ref{rigidity for F-subharmonic}, Lemma \ref{no singular point lemma} and Lemma \ref{scale of F-subharmonic}, respectively. Hence, for any $i\in I$, we have
\begin{equation*}
\theta(u,x_{i_{0}},2^{-i+2})-\theta(u,x_{i_{0}},2^{-i+3}\delta_{0})\geq\delta_{0}.
\end{equation*}
Since $F$-subharmonic function is subharmonic (see \cite[p.30]{HL1}), by Lemma \ref{p-plane lemma}, it is clear that
\begin{equation*}
\theta(u,x_{i_{0}},\frac{1}{2})-\theta(u,x_{i_{0}},0)\leq L(\Lambda,p,n),
\end{equation*}
which implies
\begin{equation}\label{estimate of singular set equation2}
|I|\leq C(L,\delta_{0}).
\end{equation}
Combining (\ref{estimate of singular set equation1}) and (\ref{estimate of singular set equation2}), we get the desired estimate.

\bigskip
\noindent
{\bf Case 2.} $u$ does not satisfies monotonicity condition.

\bigskip
By Lemma \ref{monotonicity condition lemma}, we obtain $u-N$ satisfies monotonicity condition. By Case 1, we have
\begin{equation*}
\#\left(E_{c}(u)\cap B_{1}(0^{n})\right)\leq C(c,\Lambda,F).
\end{equation*}
By the definition of $E_{c}(u)$, it is clear that $E_{c}(u)=E_{c}(u-N)$. And this completes the proof.
\end{proof}

\section{$\mathbb{G}$-plurisubharmonic functions}
In this section, we study the singular sets of $\mathbb{G}$-plurisubharmonic functions and give the proof of Theorem \ref{main theorem for G-plurisubharmonic}. Without loss of generality, we assume that $\mathbb{G}$ is a smooth submanifold of $G(p,\mathbf{R}^{n})$ (see Remark \ref{smooth remark}).

\subsection{$\mathbb{G}$-energy}
In this subsection, we introduce the $\mathbb{G}$-energies of $\mathbb{G}$-plurisubharmonic functions. And then we prove a property of $\mathbb{G}$-energy.
\begin{definition}
Let $u$ be a $\mathbb{G}$-plurisubharmonic function on $B_{2R}(0^{n})$. For any $x\in B_{R}(0^{n})$ and $r\in(0,R)$, the $\mathbb{G}$ energy of $u$ is defined by
\begin{equation*}
\theta_{\mathbb{G}}(u,x,r)=\int_{\mathbb{G}}\frac{S_{-}'(u|_{W+x},x,r)}{K_{p}'(r)}dW+\int_{\mathbb{G}}\frac{M_{-}'(u|_{W+x},x,r)}{K_{p}'(r)}dW+\frac{M_{-}'(u,x,r)}{K_{p}'(r)},
\end{equation*}
where $K_{p}$ is the Riesz kernel (see \cite[(1.1)]{HL1}). And we define $\theta_{\mathbb{G}}(u,x,0)=\lim_{r\rightarrow0}\theta_{\mathbb{G}}(u,x,r)$.
\end{definition}
Since $u$ is $\mathbb{G}$-plurisubharmonic, $S(u|_{W+x},x,\cdot)$, $M(u|_{W+x},x,\cdot)$ are $K_{p}$-convex for any $W\in\mathbb{G}$ and $x\in B_{1}(0^{n})$. It is clear that $\theta_{\mathbb{G}}(u,x,r)$ is nondecreasing function in $r$.

\begin{lemma}\label{fiber bundle control lemma}
Let $u$ be a $\mathbb{G}$-plurisubharmonic function on $B_{R}(0^{n})$. Then for any $0<a<b<R$, there exists constant $C(a,b,\mathbb{G})$ such that
\begin{equation*}
\int_{\mathbb{G}}\|u|_{W}\|_{L^{1}(A_{a,b}\cap W)}dW\leq C\|u\|_{L^{1}(A_{a,b})},
\end{equation*}
where $A_{a,b}=\{x\in\mathbf{R}^{n}|a\leq|x|\leq b\}$.
\end{lemma}

\begin{proof}
For any $0<a<b<R$, we define
\begin{equation*}
E_{a,b}:=\{(W,x)\in\mathbb{G}\times A_{a,b}~|~x\in W\}.
\end{equation*}
Thus, $E_{a,b}\overset{\sigma}{\longrightarrow}\mathbb{G}$ and $E_{a,b}\overset{\pi}{\longrightarrow}A_{a,b}$ are fiber bundles, where $\sigma$ and $\pi$ are projections onto the first and second factor (see \cite[p.7]{HL2}). We consider the pull back function $\pi^{\star}u$ on $E_{a,b}$. Since the fiber bundle is locally a product space, then there exists constants $C_{\sigma}(a,b,\mathbb{G})$ and $C_{\pi}(a,b,\mathbb{G})$ such that
\begin{equation*}
\begin{split}
\int_{\mathbb{G}}\|u|_{W}\|_{L^{1}(A_{a,b}\cap W)}dW
&  =  \int_{\mathbb{G}}\int_{A_{a,b}\cap W}|u|_{W}(x)|dxdW\\
& \leq C_{\sigma}\int_{E_{a,b}}|u|_{W}(x)|dV_{E_{a,b}}\\
&  =  C_{\sigma}\int_{E_{a,b}}|(\pi^{\star}u)(W,x)|dV_{E_{a,b}}\\
& \leq C_{\sigma}C_{\pi}\int_{A_{a,b}}|u(x)|dx,
\end{split}
\end{equation*}
where $dV_{E_{a,b}}$ is the volume form on $E_{a,b}$.
\end{proof}

\begin{lemma}\label{energy bound lemma}
Let $u$ be a $\mathbb{G}$-plurisubharmonic function on $B_{2}(0^{n})$ with $\|u\|_{L^{1}(B_{2}(0^{n}))}\leq\Lambda$. Then for any $x\in B_{1}(0^{n})$, there exists constant $C(\mathbb{G})$ such that
\begin{equation*}
\theta_{\mathbb{G}}(u,x,\frac{1}{2})\leq C(\mathbb{G})\Lambda.
\end{equation*}
\end{lemma}

\begin{proof}
Since $S(u|_{W+x},x,\cdot)$ and $M(u|_{W+x},x,\cdot)$ are $K_{p}$-convex, we have
\begin{equation}\label{energy bound lemma equation1}
\begin{split}
\theta_{\mathbb{G}}(u,x,\frac{1}{2})=& \int_{\mathbb{G}}\frac{S_{-}'(u|_{W+x},x,\frac{1}{2})}{K_{p}'(\frac{1}{2})}dW+\int_{\mathbb{G}}\frac{M_{-}'(u|_{W+x},x,\frac{1}{2})}{K_{p}'(\frac{1}{2})}dW
+\frac{M_{-}'(u,x,\frac{1}{2})}{K_{p}'(\frac{1}{2})}\\
\leq & \int_{\mathbb{G}}\frac{S(u|_{W+x},x,\frac{2}{3})-S(u|_{W+x},x,\frac{1}{2})}{K_{p}(\frac{2}{3})-K_{p}(\frac{1}{2})}dW
+\int_{\mathbb{G}}\frac{M(u|_{W+x},x,\frac{2}{3})-M(u|_{W+x},x,\frac{1}{2})}{K_{p}(\frac{2}{3})-K_{p}(\frac{1}{2})}dW\\
& +\frac{M(u,x,\frac{2}{3})-M(u,x,\frac{1}{2})}{K_{p}(\frac{2}{3})-K_{p}(\frac{1}{2})}.
\end{split}
\end{equation}
By the submean value property of subharmonic functions, it is clear that
\begin{equation}\label{energy bound lemma equation2}
S(u|_{W+x},x,\frac{2}{3})\leq M(u|_{W+x},x,\frac{2}{3})\leq\frac{3^{p}}{\omega_{p}}\|u|_{W+x}\|_{L^{1}((A_{\frac{1}{3},1}\cap W)+x)},
\end{equation}
where $\omega_{p}$ is the volume of unit ball in $\mathbb{R}^{p}$. Combining (\ref{energy bound lemma equation1}), (\ref{energy bound lemma equation2}), Lemma \ref{fiber bundle control lemma} and Lemma \ref{p-plane lemma}, we obtain
\begin{equation*}
\begin{split}
\theta_{\mathbb{G}}(u,x,\frac{1}{2}) \leq& \int_{\mathbb{G}}\frac{S(u|_{W+x},x,\frac{2}{3})-S(u|_{W+x},x,\frac{1}{2})}{K_{p}(\frac{2}{3})-K_{p}(\frac{1}{2})}dW
+\int_{\mathbb{G}}\frac{M(u|_{W+x},x,\frac{2}{3})-M(u|_{W+x},x,\frac{1}{2})}{K_{p}(\frac{2}{3})-K_{p}(\frac{1}{2})}dW\\
& +\frac{M(u,x,\frac{2}{3})-M(u,x,\frac{1}{2})}{K_{p}(\frac{2}{3})-K_{p}(\frac{1}{2})}\\
\leq& ~C\int_{\mathbb{G}}\|u|_{W+x}\|_{L^{1}((A_{\frac{1}{3},1}\cap W)+x)}dW+C\Lambda\\
\leq& ~C~\|u\|_{L^{1}(A_{\frac{1}{3},1}+x)}+C\Lambda\\
\leq& ~C\Lambda,
\end{split}
\end{equation*}
where $C$ depends only on $\mathbb{G}$.
\end{proof}

\subsection{Quantitative rigidity theorem}
In this subsection, we prove quantitative rigidity theorem of $\mathbb{G}$-plurisubharmonic functions.
\begin{lemma}\label{p-plane converge lemma}
Let $\{u_{i}\}$ be a sequence of $\mathbb{G}$-plurisubharmonic functions on $B_{R}(0^{n})$ with $\|u_{i}\|_{L^{1}(B_{R}(0^{n}))}\leq\Lambda$. Then there exists a subsequence $\{u_{i_{k}}\}$ such that $u_{i_{k}}$ converge to $u$ in $L_{loc}^{1}(B_{R}(0^{n}))$, where $u$ is a $\mathbb{G}$-plurisubharmonic function. And for almost every $W\in\mathbb{G}$, $u_{i_{k}}$ converge to $u$ in $L^{1}(A_{a,b}\cap W)$ for any $0<a<b<R$. In particular, for every $r\in(0,R)$, we have
\begin{equation*}
\lim_{k\rightarrow\infty}S(u_{i_{k}}|_{W},0^{p},r)=S(u|_{W},0^{p},r)
\end{equation*}
and
\begin{equation*}
\lim_{k\rightarrow\infty}M(u_{i_{k}}|_{W},0^{p},r)=M(u|_{W},0^{p},r)
\end{equation*}
for almost every $W\in\mathbb{G}$, where $0^{p}$ is the origin in $\mathbf{R}^{p}$.
\end{lemma}

\begin{proof}
By Lemma \ref{F subharmonic converge lemma}, there exists a subsequence $\{u_{i_{k}}\}$ such that $u_{i_{k}}$ converge to $u$ in $L_{loc}^{1}(B_{R}(0^{n}))$, where $u$ is a $\mathbb{G}$-plurisubharmonic function.

For any $0<a<b<R$, recalling $E_{a,b}\overset{\pi}{\longrightarrow}A_{a,b}$ is a fiber bundle, we consider the pull back function $\pi^{\star}u_{i_{k}}$ and $\pi^{\star}u$ on $E_{a,b}$. Since $u_{i_{k}}$ converges to $u$ in $L^{1}(A_{a,b})$, we have $\pi^{\star}u_{i_{k}}$ converge to $\pi^{\star}u$ in $L^{1}(E_{a,b})$, i.e.,
\begin{equation*}
\lim_{k\rightarrow\infty}\int_{E_{a,b}}|\pi^{\star} u_{i_{k}}-\pi^{\star}u|=0,
\end{equation*}
which implies
\begin{equation*}
\lim_{k\rightarrow\infty}\int_{\mathbb{G}}\int_{A_{a,b}\cap W}|u_{i_{k}}(x)-u(x)|dxdW=0.
\end{equation*}
By Fatou's Lemma, we have
\begin{equation*}
\int_{\mathbb{G}}\lim_{k\rightarrow\infty}\int_{A_{a,b}\cap W}|u_{i_{k}}(x)-u(x)|dxdW\leq\lim_{k\rightarrow\infty}\int_{\mathbb{G}}\int_{A_{a,b}\cap W}|u_{i_{k}}(x)-u(x)|dxdW=0.
\end{equation*}
Thus, for almost every $W\in\mathbb{G}$, we obtain
\begin{equation*}
\lim_{k\rightarrow\infty}\int_{A_{a,b}\cap W}|u_{i_{k}}(x)-u(x)|dx=0,
\end{equation*}
which implies $u_{i_{k}}|_{W}$ converge to $u|_{W}$ in $L^{1}(A_{a,b}\cap W)$. Since $u_{i_{k}}|_{W}$ and $u|_{W}$ are subharmonic functions on $A_{a,b}\cap W$, for any $r\in(a,b)$, by Lemma \ref{M S stable lemma}, we obtain
\begin{equation*}
\lim_{k\rightarrow\infty}S(u_{i_{k}}|_{W},0^{p},r)=S(u|_{W},0^{p},r)
\end{equation*}
and
\begin{equation*}
\lim_{k\rightarrow\infty}M(u_{i_{k}}|_{W},0^{p},r)=M(u|_{W},0^{p},r)
\end{equation*}
for almost every $W\in\mathbb{G}$.
\end{proof}

In order to prove quantitative rigidity theorem, we split up into different cases. First, we consider the case $p>2$.

\begin{theorem}[Quantitative rigidity theorem, $p>2$]\label{quantitative rigidity theorem, p>2}
For any $\epsilon,\lambda>0$, there exists constant $\delta_{0}(\epsilon,\lambda,\mathbb{G})$ such that if $u$ is a $\mathbb{G}$-plurisubharmonic function on $B_{\delta_{0}^{-1}}(0^{n})$ and satisfies
\begin{enumerate}[~~~~~~(1)]
    \item $\|u\|_{L^{1}(B_{r}(0^{n}))}\leq \lambda r^{n-p+2}$, for any $r\in(0,\delta_{0}^{-1})$;
    \item $\theta_{\mathbb{G}}(u,0^{n},\delta_{0}^{-1})-\theta_{\mathbb{G}}(u,0^{n},\delta_{0})\leq \delta_{0}$,
\end{enumerate}
then $u$ is $(0,\epsilon,1,0^{n})$-homogeneous.
\end{theorem}

\begin{proof}
We argue by contradiction, assuming that there exists a sequence of $\mathbb{G}$-plurisubharmonic functions $u_{i}$ on $B_{i}(0^{n})$ such that
\begin{enumerate}[~~~~~~(1)]
    \item $\|u_{i}\|_{L^{1}(B_{r}(0^{n}))}\leq \lambda r^{n-p+2}$, for any $r\in(0,i)$;
    \item $\theta_{\mathbb{G}}(u_{i},0^{n},i)-\theta_{\mathbb{G}}(u_{i},0^{n},i^{-1})\leq i^{-1}$;
    \item $u_{i}$ is not $(0,\epsilon,1,0^{n})$-homogeneous.
\end{enumerate}
By Lemma \ref{p-plane converge lemma}, there exists a subsequence $\{u_{i_{k}}\}$ such that $u_{i_{k}}$ converge to $u$ in $L_{loc}^{1}(\mathbf{R}^{n})$, where $u$ is a $\mathbb{G}$-plurisubharmonic function on $\mathbf{R}^{n}$. And for any $r>0$, we have
\begin{equation*}
\lim_{k\rightarrow\infty}S(u_{i_{k}}|_{W},0^{p},r)=S(u|_{W},0^{p},r)
\end{equation*}
and
\begin{equation*}
\lim_{k\rightarrow\infty}M(u_{i_{k}}|_{W},0^{p},r)=M(u|_{W},0^{p},r)
\end{equation*}
for almost every $W\in\mathbb{G}$.

Since $S(u|_{W},0^{p},\cdot)$ and $M(u|_{W},0^{p},\cdot)$ are $K_{p}$-convex functions, combining this with Fatou's Lemma, Lemma \ref{p-plane converge lemma} and Lemma \ref{convex converge lemma}, for almost any $r>t>0$, we obtain
\begin{equation*}
\begin{split}
\theta_{\mathbb{G}}(u,0^{n},r)-\theta_{\mathbb{G}}(u,0^{n},t)
= &  \int_{\mathbb{G}} \lim_{k\rightarrow\infty} \left(\frac{S_{-}'(u_{i_{k}}|_{W},0^{p},r)}{K_{p}'(r)}-\frac{S_{-}'(u_{i_{k}}|_{W},0^{p},t)}{K_{p}'(t)}\right)dW\\
  & +\int_{\mathbb{G}} \lim_{k\rightarrow\infty} \left(\frac{M_{-}'(u_{i_{k}}|_{W},0^{p},r)}{K_{p}'(r)}-\frac{M_{-}'(u_{i_{k}}|_{W},0^{p},t)}{K_{p}'(t)}\right)dW\\
  & +\lim_{k\rightarrow\infty} \left(\frac{M_{-}'(u_{i_{k}},0^{p},r)}{k_{p}'(r)}-\frac{M_{-}'(u_{i_{k}},0^{p},t)}{k_{p}'(t)}\right)\\
\leq & \int_{\mathbb{G}} \lim_{k\rightarrow\infty}\left(\frac{S_{-}'(u_{i_{k}}|_{W},0^{p},i_{k})}{K_{p}'(i_{k})}-\frac{S_{-}'(u_{i_{k}}|_{W},0^{p},i_{k}^{-1})}{K_{p}'(i_{k}^{-1})}\right)dW\\
& +\int_{\mathbb{G}} \lim_{k\rightarrow\infty}\left(\frac{M_{-}'(u_{i_{k}}|_{W},0^{p},i_{k})}{K_{p}'(i_{k})}-\frac{M_{-}'(u_{i_{k}}|_{W},0^{p},i_{k}^{-1})}{K_{p}'(i_{k}^{-1})}\right)dW\\
& +\lim_{k\rightarrow\infty} \left(\frac{M_{-}'(u_{i_{k}},0^{p},i_{k})}{k_{p}'(i_{k})}-\frac{M_{-}'(u_{i_{k}},0^{p},i_{k}^{-1})}{k_{p}'(i_{k}^{-1})}\right)\\
\leq& \lim_{k\rightarrow\infty}\left(\theta_{\mathbb{G}}(u_{i_{k}},0^{n},i_{k})-\theta_{\mathbb{G}}(u_{i_{k}},0^{n},i_{k}^{-1})\right)\\
\leq& ~~0.
\end{split}
\end{equation*}
By the monotonicity of $\theta_{\mathbb{G}}(u,0^{n},\cdot)$, we have
\begin{equation*}
\theta_{\mathbb{G}}(u,0^{n},r)=\theta_{\mathbb{G}}(u,0^{n},0),
\end{equation*}
for any $r>0$. It then follows that
\begin{equation}\label{quantitative rigidity theorem equation2}
S(u|_{W},0^{p},r)=\Theta(u|_{W},0^{p})K_{p}(r)+C_{S}(W)
\end{equation}
and
\begin{equation}\label{quantitative rigidity theorem equation3}
M(u|_{W},0^{p},r)=\Theta(u|_{W},0^{p})K_{p}(r)+C_{M}(W)
\end{equation}
for almost every $W\in\mathbb{G}$, where $\Theta(u|_{W},0^{p})=\Theta^{S}(u|_{W},0^{p})=\Theta^{M}(u|_{W},0^{p})$ (see \cite[(12.3)]{HL1}). By (\ref{quantitative rigidity theorem equation2}), for any $b>a>0$, we obtain
\begin{equation*}
\begin{split}
\int_{\left(B_{b}(0^{n})\setminus B_{a}(0^{n})\right)\cap W}\Delta(u|_{W}) =& \int_{B_{b}(0^{n})\cap W}\Delta(u|_{W})-\int_{B_{a}(0^{n})\cap W}\Delta(u|_{W})\\
=& ~ C(p)\left(\frac{S_{-}'(u|_{W},0^{p},b)}{K_{p}'(b)}-\frac{S'_{-}(u|_{W},0^{p},a)}{K_{p}'(a)}\right)\\
=& ~~0,
\end{split}
\end{equation*}
where we used $S_{-}'(u|_{W},0^{n},r)=C(p)K_{p}'(r)\int_{B_{r}(0^{n})}\Delta(u|_{W})$ for any $r>0$ (see e.g. \cite[Theorem 3.2.16]{Ho} or \cite[p.33]{HL1}). It then follows that $u|_{W}$ is harmonic on $W\setminus\{0^{p}\}$. By Harnack's inequality and (\ref{quantitative rigidity theorem equation3}), it is clear that
\begin{equation}\label{quantitative rigidity theorem equation4}
\limsup_{x\rightarrow0^{p}}|x|^{p-2}|u|_{W}(x)|<+\infty.
\end{equation}
Combining Theorem 10.5 in \cite{GTM137} and (\ref{quantitative rigidity theorem equation4}), we get
\begin{equation}\label{quantitative rigidity theorem equation5}
u|_{W}(x)=\Theta(u|_{W},0^{p})K_{p}(|x|)+h_{W}(x)
\end{equation}
on $W$, where $h_{W}$ is a harmonic function on $W$. By (\ref{quantitative rigidity theorem equation3}) and (\ref{quantitative rigidity theorem equation5}), we have
\begin{equation*}
M(h_{W},0^{p},r)=C_{M}(W),
\end{equation*}
for any $r>0$. By Strong Maximum Principle, we conclude that $h_{W}=C_{M}(W)$. It then follows that $u|_{W}=\Theta(u|_{W},0^{p})K_{p}+C_{M}(W)$ for almost every $W\in\mathbb{G}$. Combining Lemma \ref{fiber bundle control lemma} and (1), by scaling, we obtain
\begin{equation*}
\int_{\mathbb{G}}\int_{A_{r,2r}\cap W}|u|_{W}|dW\leq C(\mathbb{G})r^{p-n}\int_{A_{r,2r}}|u(x)|dx\leq C(\mathbb{G})\lambda r^{2},
\end{equation*}
which implies
\begin{equation*}
\int_{\mathbb{G}}\int_{A_{r,2r}\cap W}|-\Theta(u|_{W},0^{p})|x|^{2-p}+C_{M}(W)|dW\leq C(\mathbb{G})\lambda r^{2}.
\end{equation*}
It then follows that
\begin{equation*}
\left(\int_{\mathbb{G}}|C_{M}(W)|dW\right)r^{p}\leq C(\mathbb{G})\left(\int_{\mathbb{G}}\Theta(u|_{W})dW+\lambda\right)r^{2}.
\end{equation*}
Since $p>2$ and $r$ is arbitrary, we have
\begin{equation*}
\int_{\mathbb{G}}|C_{M}(W)|dW=0.
\end{equation*}
Therefore, it is clear that $u|_{W}=\Theta(u|_{W},0^{p})K_{p}$ for almost every $W\in\mathbb{G}$. Recalling $u$ is a subharmonic function on $\mathbf{R}^{n}$, we get $u$ is $0$-homogeneous. However, $u_{i_{k}}$ converge to $u$ in $L_{loc}^{1}(B_{2}(0^{n}))$. Then $u_{i_{k}}$ are $(0,\epsilon,1,0^{n})$-homogeneous when $k$ is sufficiently large, which is a contradiction.
\end{proof}

Next, we prove quantitative rigidity theorem for the case $p=2$. First, we need the following lemma.

\begin{lemma}\label{p=2 density lemma}
Let $u$ be a $\mathbb{G}$-subharmonic function on $B_{2}(0^{n})$. If $p=2$, then
\begin{equation*}
\Theta^{M}(u|_{W},0^{2})=\Theta^{M}(u,0^{n}),
\end{equation*}
for almost every $W\in\mathbb{G}$.
\end{lemma}

\begin{proof}
Let $\varphi$ be a tangent to $u$ at $0^{n}$. Then there exists a sequence $\{r_{i}\}$ such that $\lim_{i\rightarrow\infty}r_{i}=0$ and $u_{0^{n},r_{i}}$ converge to $\varphi$ in $L_{loc}^{1}(\mathbf{R}^{n})$. For almost every $W\in\mathbb{G}$. By Lemma \ref{p-plane converge lemma}, we obtain that $u_{0^{n},r_{i}}|_{W}$ converge to $\varphi|_{W}$ in $L^{1}(A_{1,2}\cap W)$. On the other hand, for any non-polar plane $W\in\mathbb{G}$ (for definition of non-polar plane, see \cite[p.5]{HL2}), by passing to a subsequence, we can assume $(u|_{W})_{0^{2},r_{i}}$ converge to $\psi$ in $L_{loc}^{1}(\mathbf{R}^{2})$, where $\psi\in T_{0^{2}}(u|_{W})$. By the definition of the tangential 2-flow, it is clear that
\begin{equation*}
(u_{0^{n},r_{i}})|_{W}(x)-(u|_{W})_{0^{2},r_{i}}(x)=M(u|_{W},0^{2},r_{i})-M(u,0^{n},r_{i}),
\end{equation*}
for almost every $x\in A_{1,2}\cap W$. Since the left hand side converges to $(\varphi|_{W}-\psi)$ in $L^{1}(A_{1,2}\cap W)$ and the right hand side is independent of $x$, then we obtain
\begin{equation*}
\lim_{i\rightarrow\infty}\left(M(u|_{W},0^{2},r_{i})-M(u,0^{n},r_{i})\right)=C,
\end{equation*}
where $C$ is a constant. It then follows that
\begin{equation*}
\Theta^{M}(u|_{W},0^{2})-\Theta^{M}(u,0^{n})=\lim_{i\rightarrow\infty}\left(\frac{M(u|_{W},0^{2},r_{i})}{K_{2}(r_{i})}-\frac{M(u|_{W},0^{n},r_{i})}{K_{2}(r_{i})}\right)=0,
\end{equation*}
as required.
\end{proof}

\begin{theorem}[Quantitative rigidity theorem, $p=2$]\label{quantitative rigidity theorem, p=2}
For any $\epsilon,\lambda>0$, there exists constant $\delta_{0}(\epsilon,\lambda,\mathbb{G})$ such that if $u$ is a $\mathbb{G}$-plurisubharmonic function on $B_{\delta_{0}^{-1}}(0^{n})$ and satisfies
\begin{enumerate}[~~~~~~(1)]
    \item $\|u\|_{L^{1}(B_{r}(0^{n}))}\leq\lambda r^{n}(|\log r|+1)$, for any $r\in(0,\delta_{0}^{-1})$;
    \item $M(u,0^{n},1)=0$;
    \item $\theta_{\mathbb{G}}(u,0^{n},\delta_{0}^{-1})-\theta_{\mathbb{G}}(u,0^{n},\delta_{0})\leq \delta_{0}$,
\end{enumerate}
then $u$ is $(0,\epsilon,1,0^{n})$-homogeneous.
\end{theorem}

\begin{proof}
We argue by contradiction, assuming that there exists a sequence of $\mathbb{G}$-plurisubharmonic functions $u_{i}$ on $B_{i}(0^{n})$ such that
\begin{enumerate}[~~~~~~(1)]
    \item $\|u_{i}\|_{L^{1}(B_{r}(0^{n}))}\leq\lambda r^{n}(|\log r|+1)$, for any $r\in(0,i)$;
    \item $M(u_{i},0^{n},1)=0$;
    \item $\theta_{\mathbb{G}}(u_{i},0^{n},i)-\theta_{\mathbb{G}}(u_{i},0^{n},i^{-1})\leq i^{-1}$;
    \item $u_{i}$ is not $(0,\epsilon,1,0^{n})$-homogeneous.
\end{enumerate}
By Lemma \ref{F subharmonic converge lemma}, there exists a subsequence $\{u_{i_{k}}\}$ such that $u_{i_{k}}$ converge to $u$ in $L_{loc}^{1}(\mathbf{R}^{n})$, where $u$ is a $\mathbb{G}$-plurisubharmonic function on $\mathbf{R}^{n}$. By (2) and Lemma \ref{M S stable lemma}, we obtain $M(u,0^{n},1)=0$. Combining this and the similar argument in Theorem \ref{quantitative rigidity theorem, p>2}, for any $r>0$, we have
\begin{equation*}
\theta_{\mathbb{G}}(u,0^{n},r)=\theta_{\mathbb{G}}(u,0^{n},0),
\end{equation*}
which implies
\begin{equation*}
M(u,0^{n},r)=\Theta^{M}(u,0^{n})K_{2}(r)
\end{equation*}
and
\begin{equation*}
u|_{W}=\Theta^{M}(u|_{W},0^{2})K_{2}+C_{W},
\end{equation*}
for almost every $W\in\mathbb{G}$, where $C_{W}$ is a constant on $W$. By Lemma \ref{p=2 density lemma}, we obtain
\begin{equation*}
u|_{W}=\Theta^{M}(u,0^{n})K_{2}+C_{W}.
\end{equation*}
For $x\in W$, by definition of tangential $2$-flow, it is clear that
\begin{equation*}
\begin{split}
u_{0^{n},r}(x) & = u(rx)-M(u,0^{n},r)\\
& = \Theta^{M}(u,0^{n})K_{2}(rx)+C_{W}-\Theta^{M}(u,0^{n})K_{2}(r)\\
& = u(x).
\end{split}
\end{equation*}
It then follows that $u_{0^{n},r}(x)=u(x)$ for almost every $x\in\mathbf{R}^{n}$. Since $u_{0^{n},r}$ and $u$ are subharmonic functions. We obtain that $u_{0^{n},r}=u$ for any $r>0$. Then $u$ is $0$-homogeneous. When $k$ is sufficiently large, $u_{i_{k}}$ is $(0,\epsilon,1,0^{n})$-homogeneous, which contradicts with (4).
\end{proof}

\subsection{Covering lemma and decomposition lemma}
Let $u$ be a $\mathbb{G}$-plurisubharmonic function on $B_{2}(0^{n})$ with $\|u\|_{L^{1}(B_{2}(0^{n}))}\leq\Lambda$. First, we introduce the following definitions.
\begin{definition}
For any $\epsilon>0$, $t\geq1$ and $0<r<1$, we define
\begin{equation*}
\mathcal{H}_{t,r,\epsilon}=\{x\in B_{1}(0^{n})~|~\mathcal{N}_{t}(u,B_{r}(x))>\epsilon\}
\end{equation*}
and
\begin{equation*}
\mathcal{L}_{t,r,\epsilon}=\{x\in B_{1}(0^{n})~|~\mathcal{N}_{t}(u,B_{r}(x))\leq\epsilon\},
\end{equation*}
where
\begin{equation*}
\mathcal{N}_{t}(u,B_{r}(x))=\inf\{\delta>0~|~\text{$u$ is $(0,\delta,tr,x)$-homogeneous}\}.
\end{equation*}
\end{definition}

\begin{definition}
For any $x\in B_{1}(0^{n})$ and $\gamma\in(0,1)$, we define $j$-tuple $T^{j}(x)=(T_{1}^{j}(x),T_{2}^{j}(x),\cdots,T_{j}^{j}(x))$ by
\begin{equation*}
T_{i}^{j}(x)=\left\{ \begin{array}{ll}
1~~~~~~\text{if $x\in\mathcal{H}_{\gamma^{-1},\gamma^{i},\epsilon}$}\\
0~~~~~~\text{if $x\in\mathcal{L}_{\gamma^{-1},\gamma^{i},\epsilon}$}
\end{array}\right.
\end{equation*}
for all $1\leq i\leq j$, where $\epsilon=\epsilon(\eta,\gamma,\Lambda,\mathbb{G})$ is the constant in Lemma \ref{covering lemma} and $\gamma>0$ is a constant to be determined later.
\end{definition}

\begin{definition}
For any $j$-tuple $T^{j}$, we define
\begin{equation*}
E_{T^{j}}=\{x\in B_{1}(0^{n})~|~T^{j}(x)=T^{j}\}.
\end{equation*}
\end{definition}

Next, for each $E_{T^{j}}\neq\emptyset$, we define a collection of sets $\{\mathcal{C}_{\eta,\gamma^{j}}^{k}(T^{j})\}$ by induction, where $\mathcal{C}_{\eta,\gamma^{j}}^{k}(T^{j})$ is the union of balls of radius $\gamma^{j}$. For $j=0$, we put $\mathcal{C}_{\eta,\gamma^{0}}^{k}(T^{j})=B_{1}(0^{n})$. Assume that $\mathcal{C}_{\eta,\gamma^{j-1}}^{k}(T^{j-1})$ has been defined and satisfies $\mathcal{S}_{\eta,\gamma^{j}}^{k}(u)\cap E_{T^{j}}\subset \mathcal{C}_{\eta,\gamma^{j-1}}^{k}(T^{j-1})$, where $T^{j-1}$ is the $(j-1)$-tuple obtained from $T^{j}$ by dropping the last entry. For each ball $B_{\gamma^{j-1}}(x)$ of $\mathcal{C}_{\eta,\gamma^{j-1}}^{k}(T^{j-1})$, take a minimal covering of $B_{\gamma^{j-1}}(x)\cap \mathcal{S}_{\eta,\gamma^{j}}^{k}(u)\cap E_{T^{j}}$ by balls of radius $\gamma^{j}$ with centers in $B_{\gamma^{j-1}}(x)\cap \mathcal{S}_{\eta,\gamma^{j}}^{k}(u)\cap E_{T^{j}}$. Define the union of all balls so obtained to be $\mathcal{C}_{\eta,\gamma^{j}}^{k}(T^{j})$.

\begin{lemma}\label{rescaling ball in cone-splitting lemma}
For any $x\in B_{1}(0^{n})$, $s\in(0,\frac{1}{2})$ and $r\in(0,\frac{1}{2}s^{-1})$, there exists constant $N(\Lambda,p,n)$ such that
\begin{equation*}
\int_{B_{r}(0^{n})}|u_{x,s}(y)|dy
\leq\left\{ \begin{array}{ll}
\ Nr^{n-p+2} ~~\text{~when $p>2$}\\
\ Nr^{n}(|\log r|+1) ~~\text{~when $p=2$}
\end{array}.\right.
\end{equation*}
\end{lemma}

\begin{proof}
Without loss of generality, we assume $u\leq0$ on $B_{\frac{3}{2}}(0^{n})$. When $p>2$, since $V(u,x,\cdot)$ is $K_{p}$-convex, we have
\begin{equation*}
0\leq\frac{V(u,x,1)-V(u,x,rs)}{K_{p}(1)-K_{p}(rs)}\leq\frac{V(u,x,1)-V(u,x,\frac{1}{2})}{K_{p}(1)-K_{p}(\frac{1}{2})}\leq C(\Lambda,p,n),
\end{equation*}
which implies
\begin{equation*}
\frac{V(u,x,rs)}{K_{p}(rs)}\leq\frac{V(u,x,1)}{K_{p}(rs)}+C(\Lambda,p,n)\frac{K_{p}(rs)-K_{p}(1)}{K_{p}(rs)} \leq N(\Lambda,p,n).
\end{equation*}
Since $u\leq0$ on $B_{\frac{3}{2}}(0^{n})$, it then follows that
\begin{equation*}
\int_{B_{r}(0^{n})}|u_{x,s}(y)|dy
=-\int_{B_{r}(0^{n})}u_{x,s}(y)dy
=\omega_{p}\frac{V(u,x,rs)}{K_{p}(rs)}r^{n-p+2}
\leq Nr^{n-p+2}.
\end{equation*}

When $p=2$, by similar calculations, we have
\begin{equation}\label{rescaling ball in cone-splitting lemma equation 1}
|M(u_{x,s},0^{n},r)|=\frac{M(u,x,sr)-M(u,x,s)}{K_{2}(sr)-K_{2}(s)}|\log r|\leq C(\Lambda,n)|\log r|.
\end{equation}
By Harnack's inequality (see \cite[(7.10)]{HL1}), we obtain
\begin{equation*}
S(u_{x,s},0^{n},r)\geq C\left(M(u_{x,s},0^{n},\frac{r}{2})-M(u_{x,s},0^{n},r)\right)+M(u_{x,s},0^{n},r)\geq-C|\log r|,
\end{equation*}
which implies
\begin{equation}\label{rescaling ball in cone-splitting lemma equation 2}
V(u_{x,s},0^{n},r)=n\int_{0}^{1}S(u_{x,s},0^{n},rt)t^{n-1}dt\geq-C(|\log r|+1).
\end{equation}
Combining (\ref{rescaling ball in cone-splitting lemma equation 1}) and (\ref{rescaling ball in cone-splitting lemma equation 2}), it is clear that
\begin{equation*}
\begin{split}
\int_{B_{r}(0^{n})}|u_{x,s}(y)|dy
& = \int_{B_{r}(0^{n})}(M(u_{x,s},0^{n},r)-u_{x,s}(y))dy+\int_{B_{r}(0^{n})}|M(u_{x,s},0^{n},r)|dy\\
& \leq Cr^{n}(|\log r|+1),
\end{split}
\end{equation*}
as desired.
\end{proof}

\begin{lemma}\label{apply cone-splitting lemma}
For all $\epsilon$, $\tau$, $\gamma>0$, there exists constant $\delta(\epsilon,\tau,\gamma,\Lambda,\mathbb{G})$ with the following property. For any $r\leq1$, if $x\in \mathcal{L}_{\gamma^{-1},\gamma r,\delta}(u)$, then there exists nonnegative integer $l\leq n$ such that
\begin{enumerate}[~~~~~~(1)]
    \item $u$ is $(l,\epsilon,r,x)$-homogeneous with respect to $k$-plane $V_{u,x}^{k}$;
    \item $\mathcal{L}_{\gamma^{-1},\gamma r,\delta}\cap B_{r}(x)\subset B_{\tau r}(V_{u,x}^{k})$.
\end{enumerate}
\end{lemma}

\begin{proof}
First, we define $\delta^{[l]}$ by induction. We put $\delta^{[n]}=\frac{\epsilon}{2}$. Then we define $\delta^{[l]}=\delta(\tau,\delta^{[l+1]},N(\Lambda,\mathbb{G}),\mathbb{G})$, where $\delta$ and $N$ are the constants in Lemma \ref{cone-splitting lemma} and Lemma \ref{rescaling ball in cone-splitting lemma}, respectively. Let us put $\delta<\delta^{[0]}$. Then $\delta<\delta^{[0]}\leq\delta^{[1]}\leq\cdot\cdot\cdot\leq\delta^{[n]}=\frac{\epsilon}{2}$. Since $x\in \mathcal{L}_{\gamma^{-1},\gamma r,\delta}(u)$, we have $u$ is $(0,\delta^{[0]},r,x)$-homogeneous. Then there exists a largest $l$ such that $u$ is $(l,\delta^{[l]},r,x)$-homogeneous, which implies $u_{x,r}$ is $(l,\delta^{[l]},1,0^{n})$-homogeneous at $0^{n}$.

If there exists $y\in\left(\mathcal{L}_{\gamma^{-1},\gamma r,\delta}\cap B_{r}(x)\right)\setminus B_{\tau r}(V_{u,x}^{l})$, then $\tilde{y}=\frac{1}{r}(y-x)\in B_{1}(0^{n})\setminus B_{\tau}(V_{u_{x,r},0^{n}}^{l})$ and $u_{x,r}$ is $(l,\delta^{[l]},1,\tilde{y})$-homogeneous. By Lemma \ref{cone-splitting lemma}, we obtain $u_{x,r}$ is $(l+1,\delta^{[l+1]},1,0^{n})$-homogeneous, which implies $u$ is $(l+1,\delta^{[l+1]},r,x)$-homogeneous, which contradicts with our assumption that $l$ is the largest one.
\end{proof}

\begin{lemma}[Covering lemma]\label{covering lemma}
There exists constant $\epsilon(\eta,\gamma,\Lambda,\mathbb{G})$ such that if $x\in\mathcal{L}_{\gamma^{-1},\gamma^{j},\epsilon}$ and $B_{\gamma^{j-1}}(x)$ is a ball of $\mathcal{C}_{\eta,\gamma^{j-1}}^{k}(T^{j-1})$, then the number of balls in the minimal covering of $B_{\gamma^{j-1}}(x)\cap\mathcal{S}_{\eta,\gamma^{j}}^{k}(u)\cap\mathcal{L}_{\gamma^{-1},\gamma^{j},\epsilon}$ is less than $C(n)\gamma^{-k}$.
\end{lemma}

\begin{proof}
We put $\epsilon=\delta(\eta,\tau,\gamma,\Lambda,\mathbb{G})$, where $\delta$ is the constant in Lemma \ref{apply cone-splitting lemma}. Since $x\in\mathcal{L}_{\gamma^{-1},\gamma^{j},\epsilon}$, by Lemma \ref{apply cone-splitting lemma}, there exists nonnegative integer $l\leq n$ such that
\begin{enumerate}[~~~~~~(1)]
    \item $u$ is $(l,\eta,\gamma^{j-1},x)$-homogeneous with respect to $k$-plane $V_{u,x}^{k}$;
    \item $\mathcal{L}_{\gamma^{-1},\gamma^{j},\eta}\cap B_{\gamma^{j-1}}(x)\subset B_{\tau\gamma^{j-1}}(V_{u,x}^{k})$.
\end{enumerate}
Since $x\in\mathcal{S}_{\eta,\gamma^{j}}^{k}(u)$, we obtain that $u$ is not $(k+1,\eta,\gamma^{j-1},x)$-homogeneous, which implies $l\leq k$. Hence, by choosing $\tau=\frac{\gamma}{10}$, we have
\begin{equation*}
B_{\gamma^{j-1}}(x)\cap\mathcal{S}_{\eta,\gamma^{j}}^{k}(u)\cap\mathcal{L}_{\gamma^{-1},\gamma^{j},\epsilon}\subset B_{\gamma^{j-1}}(x)\cap B_{\frac{\gamma^{j}}{10}}(V_{u,x}^{k}).
\end{equation*}
This completes the proof.
\end{proof}

\begin{lemma}[Decomposition lemma]\label{number of covering lemma}
There exists constants $C_{0}(n)$, $C_{1}(n)$, $K(\eta,\gamma,\Lambda,\mathbb{G})$, $Q(\eta,\gamma,\Lambda,\mathbb{G})$ and $\gamma_{0}(\eta,\Lambda,\mathbb{G})$ such that for any $\gamma<\gamma_{0}$ and $j\in\mathbb{Z}_{+}$, we have
\begin{enumerate}[~~~~~~(1)]
    \item The set $\mathcal{S}_{\eta,\gamma^{j}}^{k}(u)\cap B_{1}(0^{n})$ can be covered by at most $j^{K}$ nonempty sets $\mathcal{C}_{\eta,\gamma^{j}}^{k}$.
    \item Each set $\mathcal{C}_{\eta,\gamma^{j}}^{k}$ is the union of at most $(C_{1}\gamma^{-n})^{Q}\cdot(C_{0}\gamma^{-k})^{j-Q}$ balls of radius $\gamma^{j}$.
\end{enumerate}
\end{lemma}

\begin{proof}
First, we prove (1). We need to prove $|T^{j}|\leq K(\eta,\gamma,\Lambda,\mathbb{G})$ if $E_{T^{j}}\neq\emptyset$. For any $0<s<t<1$ and $x\in B_{1}(0^{n})$, we define
\begin{equation*}
\mathcal{W}_{s,t}(x)=\theta_{\mathbb{G}}(u,x,t)-\theta_{\mathbb{G}}(u,x,s)\geq0.
\end{equation*}
Fixing a point $x_{0}\in E_{T^{j}}$, we consider the set
\begin{equation*}
I=\{i\in\mathbb{Z}_{+}~|~\mathcal{W}_{\gamma^{i},\gamma^{i-2}}(x_{0})\geq\delta_{0}\},
\end{equation*}
where $\delta_{0}$ is the constant in Theorem \ref{quantitative rigidity theorem, p>2} ($p>2$) or Theorem \ref{quantitative rigidity theorem, p=2} ($p=2$). It is clear that
\begin{equation*}
\sum_{i\in I}\mathcal{W}_{\gamma^{i},\gamma^{i-2}}(x_{0})\leq3\mathcal{W}_{0,1}(x_{0}).
\end{equation*}
By Lemma \ref{energy bound lemma}, we have
\begin{equation*}
|I|\cdot\delta_{0}\leq3C(\mathbb{G})\Lambda.
\end{equation*}
For any $i\notin I$, by $\mathcal{W}_{\gamma^{i},\gamma^{i-2}}(x_{0})\leq\delta_{0}$, we have
\begin{equation}\label{number of covering lemma equation1}
\theta(u_{x_{0},\gamma^{i-1}},0^{n},\gamma^{-1})-\theta(u_{x_{0},\gamma^{i-1}},0^{n},\gamma)=\mathcal{W}_{\gamma^{i},\gamma^{i-2}}(x_{0})<\delta_{0}.
\end{equation}
Now, we put $\gamma_{0}=\delta_{0}$. Thus, if $\gamma<\gamma_{0}$, combining (\ref{number of covering lemma equation1}), Theorem \ref{quantitative rigidity theorem, p>2} ($p>2$), Theorem \ref{quantitative rigidity theorem, p=2} ($p=2$), Lemma \ref{rescaling ball in cone-splitting lemma} and $M(u_{x_{0},\gamma^{i-1}},0^{n},1)=0$ when $p=2$, we obtain $u_{x_{0},\gamma^{i-1}}$ is $(0,\epsilon,1,0^{n})$-homogeneous, which implies $u$ is $(0,\epsilon,\gamma^{i-1},x_{0})$-homogeneous. Hence, we have $x_{0}\in\mathcal{L}_{\gamma^{-1},\gamma^{i},\epsilon}$, which implies $T_{i}^{j}(x_{0})=0$. It then follows that there exists constant $K$ depending only on $\mathbb{G}$ and $\Lambda$ such that
\begin{equation*}
|T^{j}|:=\sum_{i=1}^{j}T_{i}^{j}\leq |I| \leq K,
\end{equation*}
which implies the cardinality of $\{\mathcal{C}_{\eta,\gamma^{j}}^{k}(T^{j})\}$ is at most
\begin{equation*}
{j\choose K} \leq j^{K}.
\end{equation*}
This proves (1).

Next, we prove (2). Clearly, by an induction argument, (2) is an immediate corollary of Lemma \ref{covering lemma}.
\end{proof}

\subsection{Proof of Theorem \ref{main theorem for G-plurisubharmonic}}
In this subsection, we give the proof of Theorem \ref{main theorem for G-plurisubharmonic}.
\begin{proof}[Proof of Theorem \ref{main theorem for G-plurisubharmonic}]
First, we put $\gamma=\min(\gamma_{0},C_{0}^{-\frac{2}{\eta}})$, where $\gamma_{0}$ and $C_{0}$ are the constants in Lemma \ref{number of covering lemma}. Clearly, it suffices to prove (\ref{main theorem for G-plurisubharmonic equation}) when $r<\gamma$. There exists $j\in\mathbb{Z}_{+}$ such that $\gamma^{j+1}\leq r<\gamma^{j}$. By Lemma \ref{number of covering lemma}, $\mathcal{S}_{\eta,\gamma^{j}}^{k}(u)\cap B_{1}(0^{n})$ can be covered by $j^{K}(C_{1}\gamma^{-n})^{Q}(C_{0}\gamma^{-k})^{j-Q}$ balls of radius $\gamma^{j}$, which implies
\begin{equation*}
\begin{split}
\text{Vol}(B_{\gamma^{j}}(\mathcal{S}_{\eta,\gamma^{j}}^{k}(u))\cap B_{1}(0^{n})) & \leq j^{K}(C_{1}\gamma^{-n})^{Q}(C_{0}\gamma^{-k})^{j-Q}(2\gamma^{j})^{n}\\
& \leq C(n,Q,K)(\gamma^{j})^{n-k-\eta}.
\end{split}
\end{equation*}
Since $\gamma^{j+1}\leq r<\gamma^{j}$, we have $\mathcal{S}_{\eta,r}^{k}(u)\subset\mathcal{S}_{\eta,\gamma^{j}}^{k}(u)$, which implies
\begin{equation*}
\begin{split}
\text{Vol}(B_{r}(\mathcal{S}_{\eta,r}^{k}(u))\cap B_{1}(0^{n}))
& \leq\text{Vol}(B_{\gamma^{j}}(\mathcal{S}_{\eta,\gamma^{j}}^{k}(u))\cap B_{1}(0^{n}))\\
& \leq C(n,Q,K)(\gamma^{j})^{n-k-\eta}\\
& \leq C(\eta,\Lambda,\mathbb{G})r^{n-k-\eta},
\end{split}
\end{equation*}
as desired.
\end{proof}

\section{Appendix}

\subsection{Homogeneous functions}
In this subsection, we assume that homogeneity of tangents holds for $F$ and Riesz characteristic $p_{F}\geq2$. In Lemma \ref{homogeneous converge lemma}, we prove a basic property of homogeneous functions. By using this property, we give the proof of (\ref{quantitative stratification}).

\begin{lemma}\label{homogeneous converge lemma}
Let $h_{i}$ be a sequence of functions on $\mathbf{R}^{n}$. Suppose that $h_{i}$ is $k$-homogeneous at $y_{i}$ with respect to $k$-plane $V_{i}^{k}$. If $\lim_{i\rightarrow\infty}y_{i}=y$, $\lim_{i\rightarrow\infty}V_{i}^{k}=V^{k}$ and $h_{i}$ converge to $u$ in $L^{1}(B_{r}(0^{n}))$. Then there exists a function $h$ such that
\begin{enumerate}[~~~~~~(1)]
    \item $h$ is defined on $\mathbf{R}^{n}$;
    \item $h$ is $k$-homogeneous at $y$ with respect to $k$-plane $V^{k}$;
    \item $h=u$ in $B_{r}(0^{n})$.
\end{enumerate}
\end{lemma}

\begin{proof}
Without loss of generality, we assume $y=0^{n}$ and $r=1$. We split up in to different cases.

\bigskip
\noindent
{\bf Case 1.} For any $i$, we have $y_{i}=0^{n}$ and $V_{i}^{k}=V^{k}$.

\bigskip
When $p=2$, for any $R>1$, we have
\begin{equation*}
\begin{split}
\int_{B_{R}(0^{n})}|h_{i}(x)-h_{j}(x)|dx & = \int_{B_{R}(0^{n})}|(h_{i})_{0^{n},\frac{1}{R}}(x)-(h_{j})_{0^{n},\frac{1}{R}}(x)|dx\\
& \leq \int_{B_{R}(0^{n})}|h_{i}(\frac{x}{R})-h_{j}(\frac{x}{R})|dx+\omega_{n}R^{n}|M(h_{i},0^{n},\frac{1}{R})-M(h_{j},0^{n},\frac{1}{R})|\\
& \leq R^{n}\|h_{i}-h_{j}\|_{L^{1}(B_{1}(0^{n}))}+\omega_{n}R^{n}|M(h_{i},0^{n},\frac{1}{R})-M(h_{j},0^{n},\frac{1}{R})|.
\end{split}
\end{equation*}
By Lemma \ref{M S stable lemma}, we obtain
\begin{equation}\label{homogeneous converge lemma equation 1}
\lim_{i,j\rightarrow\infty}\|h_{i}-h_{j}\|_{L^{1}(B_{R}(0^{n}))}=0.
\end{equation}
On the other hand, when $p>2$, by the similar argument, we still have (\ref{homogeneous converge lemma equation 1}).

Next, by (\ref{homogeneous converge lemma equation 1}), $h_{i}$ is a Cauchy sequence in $L_{loc}^{1}(\mathbf{R}^{n})$. There exists a function $h$ on $\mathbf{R}^{n}$ such that $h_{i}$ converge to $h$ in $L_{loc}^{1}(\mathbf{R}^{n})$. It is clear that $h=u$ in $B_{1}(0^{n})$. Now, it suffices to prove $h$ is $k$-homogeneous at $0^{n}$ with respect to $V^{k}$. For any $r>0$, we have $(h_{i})_{0^{n},r}=h_{i}$. Letting $i\rightarrow\infty$, we obtain $h$ is $k$-homogeneous. Since $h_{i}$ is homogeneous at $0^{n}$ with respect to $V^{k}$, then for any $x\in\mathbf{R}^{n}$ and $v\in V^{k}$, by the property of subharmonic functions, we have
\begin{equation*}
\begin{split}
h(x+v)-h(x)
& = \lim_{s\rightarrow 0}\frac{1}{\omega_{n}s^{n}}\left(\int_{B_{s}(x+v)}h(y)dy-\int_{B_{s}(x)}h(y)dy\right)\\
& = \lim_{s\rightarrow 0}\lim_{i\rightarrow\infty}\frac{1}{\omega_{n}s^{n}}\left(\int_{B_{s}(x+v)}h_{i}(y)dy-\int_{B_{s}(x)}h_{i}(y)dy\right)\\
& = 0,
\end{split}
\end{equation*}
as desired.

\bigskip
\noindent
{\bf Case 2.} General case.

\bigskip
Since $\lim_{i\rightarrow\infty}V_{i}^{k}=V^{k}$, there exists a sequence of $n\times n$ orthogonal matrices $A_{i}$ such that $V_{i}^{k}=A_{i}V^{k}$ and $\lim_{i\rightarrow\infty}A_{i}=I_{n}$, where $I_{n}$ is the $n\times n$ identity matrix. We define $\tilde{h}_{i}(x)=h_{i}(A_{i}x+y_{i})$, which implies that $\tilde{h}_{i}$ is $k$-homogeneous at $0^{n}$ with respect to $V^{k}$. For any $r\in [\frac{1}{2},1)$, we compute
\begin{equation}\label{homogeneous converge lemma equation 2}
\begin{split}
&~~\int_{B_{r}(0^{n})}|\tilde{h}_{i}(x)-u(x)|dx\\
\leq &~~ \int_{B_{r}(0^{n})}|h_{i}(A_{i}x+y_{i})-u(A_{i}x+y_{i})|dx+\int_{B_{r}(0^{n})}|u(A_{i}x+y_{i})-u(x)|dx\\
\leq &~~ \int_{B_{r}(y_{i})}|h_{i}(x)-u(x)|dx+\int_{B_{r}(0^{n})}|u(A_{i}x+y_{i})-u(x)|dx\\
\rightarrow &~~ 0,
\end{split}
\end{equation}
where we used $h_{i}$ converge to $u$ in $L^{1}(B_{1}(0^{n}))$ and Lemma \ref{translation lemma}. By Case 1, (\ref{homogeneous converge lemma equation 2}) and scaling argument, for each $r\in[\frac{1}{2},1)$, there exists a function $h^{(r)}$ such that
\begin{enumerate}[~~~~~~(1)]
    \item $h^{r}$ is defined on $\mathbf{R}^{n}$;
    \item $h^{r}$ is $k$-homogeneous at $0^{n}$ with respect to $k$-plane $V^{k}$;
    \item $h^{r}=u$ in $B_{r}(0^{n})$.
\end{enumerate}
By (2) and (3), we have
\begin{equation*}
h^{(r)}=h^{(\frac{1}{2})} ~\text{~in $\mathbf{R}^{n}$}.
\end{equation*}
Hence, $h^{(\frac{1}{2})}$ is the desired function.
\end{proof}

\begin{proposition}\label{quantitative stratification proposition}
If homogeneity of tangents holds for $F$, then for any $F$-subharmonic function $u$ on $B_{2}(0^{n})$, we have
\begin{equation*}
\mathcal{S}^{k}(u)=\bigcup_{\eta}\mathcal{S}_{\eta}^{k}(u)=\bigcup_{\eta}\bigcap_{r}\mathcal{S}_{\eta,r}^{k}(u).
\end{equation*}
\end{proposition}

\begin{proof}
For any $\eta>0$, by definition, we have $\mathcal{S}_{\eta}^{k}(u)=\bigcap_{r}\mathcal{S}_{\eta,r}^{k}(u)$ and $\mathcal{S}_{\eta}^{k}(u)\subset\mathcal{S}^{k}(u)$. It suffices to prove $\mathcal{S}^{k}(u)\subset\bigcup_{\eta}\mathcal{S}_{\eta}^{k}(u)$. We argue by contradiction, assuming that $\mathcal{S}^{k}(u)\nsubseteq\bigcup_{\eta}\mathcal{S}_{\eta}^{k}(u)$. Then there exists a point $x\in B_{2}(0^{n})$ such that
\begin{enumerate}[~~~~~~(1)]
    \item $x\in\mathcal{S}^{k}(u)$;
    \item For each $i\in\mathbf{Z}_{+}$, there exists a $(k+1)$-homogeneous function $h_{i}$ and $r_{i}\in(0,1)$ such that
          \begin{equation*}
          \int_{B_{1}(0^{n})}|u_{x,r_{i}}(y)-h_{i}(y)|dy < i^{-1}.
          \end{equation*}
\end{enumerate}
By the compactness of subharmonic functions, after passing to a subsequence, we assume
\begin{equation}\label{quantitative stratification proposition equation 1}
\lim_{i\rightarrow\infty}r_{i}=r, ~~\lim_{i\rightarrow\infty}\|h_{i}-h\|_{L^{1}(B_{\frac{1}{2}}(0^{n}))}=0 ~\textit{~and~}~
\lim_{i\rightarrow\infty}\|u_{x,r_{i}}-h\|_{L^{1}(B_{\frac{1}{2}}(0^{n}))}=0.
\end{equation}

If $r=0$, by the definition of tangent (see \cite[Definition 9.3, Proposition 9.4]{HL1}), there exists $U\in T_{x}(u)$ such that $u_{x,r_{i}}$ converge to $U$ in $L_{loc}^{1}(\mathbf{R}^{n})$. Combining this and (\ref{quantitative stratification proposition equation 1}), we have $U=h$ in $B_{\frac{1}{2}}(0^{n})$. On the other hand, since homogeneity of tangents holds for $F$, $U$ is $0$-homogeneous. By Lemma \ref{homogeneous converge lemma}, there exists a $(k+1)$-plane $V^{k+1}$ such that $h$ is $(k+1)$-homogeneous with respect to $V^{k+1}$. By Definition \ref{homogeneous definition}, we get $U=h$ is $(k+1)$-homogeneous, which contradicts with $x\in\mathcal{S}^{k}(u)$.

If $r>0$, by Lemma \ref{translation lemma}, $h=u_{x,r}$ in $B_{\frac{1}{2}}(0^{n})$. By the definition of tangent set, we have $T_{x}(u)=T_{0^{n}}(u_{x,r})=T_{0^{n}}(h)$. By Lemma \ref{homogeneous converge lemma}, $h$ is a $(k+1)$-homogeneous function, which implies $T_{0^{n}}(h)=\{h\}$, which contradicts with $x\in\mathcal{S}^{k}(u)$.
\end{proof}

\subsection{$K_{p}$-convex functions}
In this subsection, we recall some properties of $K_{p}$-convex functions, where $K_{p}$ is the classical $p^{th}$ Riesz kernel (see \cite[(1.1)]{HL1}).

\begin{lemma}\label{almost implies every lemma}
Let $\{f_{i}\}$ be a sequence of $K_{p}$-convex functions on $(0,R)$. If $\lim_{i\rightarrow\infty}f_{i}(r)=f(r)$ for almost every $r\in(0,R)$, then we have $\lim_{i\rightarrow\infty}f_{i}(r)=f(r)$ for every $r\in(0,R)$.
\end{lemma}

\begin{proof}
For any $\epsilon>0$ and $r\in(0,R)$, by assumption, there exists $0<s_{1}<s_{2}<r<s_{3}<s_{4}<R$ such that
\begin{equation}\label{almost implies every lemma equation1}
\lim_{i\rightarrow\infty}f_{i}(s_{j})=f(s_{j}) \text{~for $j=1,2,3,4$.}
\end{equation}
By the definition of $K_{p}$-convex functions, for any $r_{1},r_{2}\in(s_{2},s_{3})$, we have
\begin{equation}\label{almost implies every lemma equation2}
\frac{f_{i}(s_{2})-f_{i}(s_{1})}{K_{p}(s_{2})-K_{p}(s_{1})}
\leq\frac{f_{i}(r_{2})-f_{i}(r_{1})}{K_{p}(r_{2})-K_{p}(r_{1})}
\leq\frac{f_{i}(s_{4})-f_{i}(s_{3})}{K_{p}(s_{4})-K_{p}(s_{3})}.
\end{equation}
Combining (\ref{almost implies every lemma equation1}) and (\ref{almost implies every lemma equation2}), we obtain that $f_{i}$ and $f$ are Lipschitz functions on $[s_{2},s_{3}]$ with uniform Lipschitz constant $L(s_{1},s_{2},s_{3},s_{4},f,p)$. We can choose $\tilde{r}\in(s_{2},s_{3})$ such that $|\tilde{r}-r|\leq\epsilon$ and $\lim_{i\rightarrow\infty}f_{i}(\tilde{r})=f(\tilde{r})$. It then follows that for $i$ sufficiently large, we have
\begin{equation*}
\begin{split}
|f_{i}(r)-f(r)| &\leq |f_{i}(r)-f_{i}(\tilde{r})|+|f_{i}(\tilde{r})-f(\tilde{r})|+|f(\tilde{r})-f(r)|\\
&\leq 2L\epsilon+|f_{i}(\tilde{r})-f(\tilde{r})|\\
&\leq (2L+1)\epsilon,
\end{split}
\end{equation*}
which implies $\lim_{i\rightarrow\infty}f_{i}(r)=f(r)$. Since $r$ is arbitrary, we complete the proof.
\end{proof}

\begin{lemma}\label{convex converge lemma}
Let $\{f_{i}\}$ be a sequence of $K_{p}$-convex functions on $(0,R)$. If $\lim_{i\rightarrow\infty}f_{i}(r)=f(r)$ for every $r\in(0,R)$, then we have
\begin{equation*}
\lim_{i\rightarrow\infty}\frac{(f_{i})_{\pm}'(r)}{K_{p}'(r)}=\frac{f'(r)}{K_{p}'(r)},
\end{equation*}
for almost every $r\in(0,R)$.
\end{lemma}

\begin{proof}
Since $f_{i}$ are $K_{p}$-convex functions and $f=\lim_{i\rightarrow\infty}f_{i}$, it is clear that $f$ is also $K_{p}$-convex function. As a result, we obtain $f$ is differentiable almost everywhere in $(0,R)$. For any $r_{0}\in(0,R)$ at which $f$ is differentiable and for any $\epsilon>0$, there exists $r>0$ such that
\begin{equation*}
\frac{f'(r_{0})}{K_{p}'(r_{0})}-\epsilon\leq\frac{f(r_{0})-f(r_{0}-r)}{K_{p}(r_{0})-K_{p}(r_{0}-r)}
\leq\frac{f(r_{0}+r)-f(r_{0})}{K_{p}(r_{0}+r)-K_{p}(r_{0})}\leq\frac{f'(r_{0})}{K_{p}'(r_{0})}+\epsilon.
\end{equation*}
Then there exists $N>0$ such that for any $i\geq N$, we have
\begin{equation*}
\frac{f'(r_{0})}{K_{p}(r_{0})}-2\epsilon\leq\frac{f(r_{0})-f(r_{0}-r)}{K_{p}(r_{0})-K_{p}(r_{0}-r)}-\epsilon\leq\frac{f_{i}(r_{0})-f_{i}(r_{0}-r)}{K_{p}(r_{0})-K_{p}(r_{0}-r)}\leq\frac{(f_{i})_{-}'(r_{0})}{K_{p}'(r_{0})}
\end{equation*}
and
\begin{equation*}
\frac{(f_{i})_{+}'(r_{0})}{K_{p}'(r_{0})}\leq\frac{f_{i}(r_{0}+r)-f_{i}(r_{0})}{K_{p}(r_{0}+r)-K_{p}(r_{0})}\leq\frac{f(r_{0}+r)-f(r_{0})}{K_{p}(r_{0}+r)-K_{p}(r_{0})}+\epsilon\leq\frac{f'(r_{0})}{K_{p}(r_{0})}+2\epsilon.
\end{equation*}
Combining with
\begin{equation*}
\frac{(f_{i})_{-}'(r_{0})}{K_{p}'(r_{0})}\leq\frac{(f_{i})_{+}'(r_{0})}{K_{p}'(r_{0})},
\end{equation*}
we complete the proof.
\end{proof}

\subsection{Subharmonic function in $\mathbf{R}^{p}$}
In this subsection, we recall some properties of subharmonic functions.
\begin{lemma}\label{p-plane lemma}
Let $v$ be a subharmonic function on $B_{R}(0^{p})\subset\mathbf{R}^{p}$ with $\|v\|_{L^{1}(B_{b}(0^{p})\setminus(B_{a}(0^{p}))}\leq\Lambda$, where $0<a<b<R$. Then for any $t\in(a+d,b-d)$, where $d>0$, there exists a constant $C(t,a,d)$ such that
\begin{equation*}
M(v,0^{p},t)\geq S(v,0^{p},t)\geq-C(t,a,d)\Lambda,
\end{equation*}
where $0^{p}$ is the origin in $\mathbf{R}^{p}$.
\end{lemma}

\begin{proof}
It suffices to prove $S(v,0^{p},t)\geq-C(t,a,d)\Lambda$. First, by the submean value property of subharmonic functions, we have
\begin{equation*}
\sup_{B_{b-d}(0^{p})\setminus B_{a+d}(0^{p})}v\leq\tilde{C}(d,\Lambda).
\end{equation*}
Thus, we compute
\begin{equation*}
\begin{split}
\int_{B_{t}(0^{p})\setminus B_{a+d}(0^{p})}|\tilde{C}-v(x)|dx & = \int_{B_{t}(0^{p})\setminus B_{a+d}(0^{p})}\left(\tilde{C}-v(x)\right)dx\\
& =\int_{a+d}^{t}p\omega_{p}s^{p-1}\left(\tilde{C}-S(v,0^{p},s)\right)ds\\
& \geq \left(\tilde{C}-S(v,0^{p},t)\right)\omega_{p}\left(t^{p}-(a+d)^{p}\right),
\end{split}
\end{equation*}
where $\omega_{p}$ is the volume of unit ball in $\mathbf{R}^{p}$. It is clear that
\begin{equation*}
\begin{split}
\left(\tilde{C}-S(v,0^{p},t)\right)\omega_{p}\left(t^{p}-(a+d)^{p}\right) & \leq \|\tilde{C}-v\|_{L^{1}(B_{t}(0^{p})\setminus B_{a+d}(0^{p}))}\\
& \leq \tilde{C}\omega_{p}\left(t^{p}-(a+d)^{p}\right)+\Lambda.
\end{split}
\end{equation*}
Hence, we obtain
\begin{equation*}
S(v,0^{p},t)\geq-\frac{\Lambda}{\omega_{p}\left(t^{p}-(a+d)^{p}\right)}.
\end{equation*}
\end{proof}

\begin{lemma}\label{M S stable lemma}
Let $v_{i}$ and $v$ be subharmonic functions on $B_{R}(0^{p})\subset\mathbf{R}^{p}$. If $v_{i}$ converge to $v$ in $L^{1}(B_{b}(0^{p})\setminus B_{a}(0^{p}))$, where $0<a<b<R$, then for any $r\in(a,b)$, we have
\begin{equation}\label{M S stable lemma equation1}
\lim_{i\rightarrow\infty}M(v_{i},0^{p},r)=M(v,0^{p},r)
\end{equation}
and
\begin{equation}\label{M S stable lemma equation2}
\lim_{i\rightarrow\infty}S(v_{i},0^{p},r)=S(v,0^{p},r).
\end{equation}
\end{lemma}

\begin{proof}
First, by the property of subharmonic functions, for any $x\in B_{b}(0^{p})\setminus B_{a}(0^{p})$, we have
\begin{equation*}
v_{i}(x)\leq v_{i}*\phi_{\delta}(x) \text{~and~} \lim_{i\rightarrow\infty}v_{i}*\phi_{\delta}(x)=v*\phi_{\delta}(x),
\end{equation*}
where $\phi$ is a mollifier. It then follows that
\begin{equation*}
\limsup_{i\rightarrow\infty}v_{i}(x)\leq\lim_{\delta\rightarrow0}v*\phi_{\delta}(x)=v(x),
\end{equation*}
which implies
\begin{equation}\label{M S stable lemma equation3}
\limsup_{i\rightarrow\infty}M(v_{i},0^{p},r)\leq M(v,0^{p},r).
\end{equation}
Suppose we have
\begin{equation*}
\liminf_{i\rightarrow\infty}M(v_{i},0^{p},r)<M(v,0^{p},r),
\end{equation*}
then there exists a subsequence $\{v_{i_{k}}\}$ and a number $d$ such that
\begin{equation}\label{M S stable lemma equation4}
\lim_{k\rightarrow\infty}M(v_{i_{k}},0^{p},r)=\liminf_{i\rightarrow\infty}M(v_{i},0^{p},r)<d<M(v,0^{p},r).
\end{equation}
Then we get $v_{i_{k}}\leq d$ on $B_{r}(0^{p})$ when $k$ is sufficiently large. By the convergence in $L^{1}(B_{b}(0^{p})\setminus B_{a}(0^{p}))$, we obtain $v\leq d$ on $B_{r}(0^{p})\setminus B_{a}(0^{p})$. Since $v$ is subharmonic function, we have
\begin{equation*}
M(v,0^{p},r)\leq d,
\end{equation*}
which contradicts with (\ref{M S stable lemma equation4}). Therefore, we conclude that
\begin{equation}\label{M S stable lemma equation5}
\liminf_{i\rightarrow\infty}M(v_{i},0^{p},r)\geq M(v,0^{p},r).
\end{equation}
Combining (\ref{M S stable lemma equation3}) and (\ref{M S stable lemma equation5}), we prove (\ref{M S stable lemma equation1}).

For the proof of (\ref{M S stable lemma equation2}), by Fatou's lemma, it is clear that
\begin{equation*}
\int_{a}^{b}\left(\lim_{i\rightarrow\infty}\int_{\partial B_{r}(0^{p})}|v_{i}-v|\right)dr\leq\lim_{i\rightarrow\infty}\int_{B_{b}(0^{p})\setminus B_{a}(0^{p})}|v_{i}(x)-v(x)|dx\rightarrow0,
\end{equation*}
which implies
\begin{equation*}
\lim_{i\rightarrow\infty}S(v_{i},0^{p},r)=S(v,0^{p},r)
\end{equation*}
for almost every $r\in(a,b)$. Since $S(v_{i},0^{p},\cdot)$ and $S(v,0^{p},\cdot)$ are $K_{p}$-convex functions, by Lemma \ref{almost implies every lemma}, we obtain (\ref{M S stable lemma equation2}).
\end{proof}

\begin{lemma}\label{translation lemma}
Suppose that $A_{i}$ is a sequence of $p\times p$ orthogonal matrices and $z_{i}$ is a sequence of points. Let $v$ be a subharmonic function on $B_{R}(0^{n})$. If $\lim_{i\rightarrow\infty}z_{i}=0^{n}$ and $\lim_{i\rightarrow\infty}A_{i}=I_{p}$ ($I_{p}$ is the $p\times p$ identity matrix), then for any $r\in(0,R)$, we have
\begin{equation*}
\lim_{i\rightarrow\infty}\int_{B_{r}(0^{n})}|v(A_{i}x+z_{i})-v(x)|dx=0.
\end{equation*}
\end{lemma}

\begin{proof}
For convenience, we use $v_{\delta}$ to denote $v*\phi_{\delta}$, where $\phi_{\delta}$ is a mollifier. By the property of smooth approximation, it is clear that $v_{\delta}$ converges to $v$ in $L_{loc}^{1}(B_{R}(0^{n}))$. On the other hand, since $v_{\delta}$ is smooth, we have
\begin{equation*}
\lim_{i\rightarrow\infty}\int_{B_{r}(0^{n})}|v_{\delta}(A_{i}x+z_{i})-v_{\delta}(x)|dx=0.
\end{equation*}
Therefore, we obtain
\begin{equation*}
\begin{split}
&~~ \int_{B_{r}(0^{n})}|v(A_{i}x+z_{i})-v(x)|dx\\
\leq &~~ \int_{B_{r}(0^{n})}|v(A_{i}x+z_{i})-v_{\delta}(A_{i}x+z_{i})|dx+\int_{B_{r}(0^{n})}|v_{\delta}(A_{i}x+z_{i})-v_{\delta}(x)|dx+\int_{B_{r}(0^{n})}|v_{\delta}(x)-v(x)|dx\\
\leq &~~ \int_{B_{r}(z_{i})}|v(x)-v_{\delta}(x)|dx+\int_{B_{r}(0^{n})}|v_{\delta}(A_{i}x+z_{i})-v_{\delta}(x)|dx+\int_{B_{r}(0^{n})}|v_{\delta}(x)-v(x)|dx\\
\rightarrow & ~~ 0,
\end{split}
\end{equation*}
as desired.
\end{proof}

\begin{lemma}\label{tangential p-flow stable lemma}
Let $v_{i}$ and $v$ be subharmonic functions on $B_{2}(0^{n})$, and suppose that $v_{i}$ converge to $v$ in $L_{loc}^{1}(B_{2}(0^{n}))$. For any sequence of point $\{z_{i}\}\subset B_{1}(0^{n})$, if $z_{i}$ converge to $z$, then we have
\begin{equation*}
\lim_{i\rightarrow\infty}\int_{B_{1}(0^{n})}|(v_{i})_{z_{i},r}(x)-v_{z,r}(x)|dx,
\end{equation*}
for any $r\in(0,1)$.
\end{lemma}

\begin{proof}
We split up into different cases.

\bigskip
\noindent
{\bf Case 1.} $p>2$.

\bigskip
For any $r\in(0,1)$, by Lemma \ref{translation lemma}, we have
\begin{equation*}
\begin{split}
&~~ \int_{B_{1}(0^{n})}|(v_{i})_{z_{i},r}(x)-v_{z,r}(x)|dx\\
\leq &~~ \int_{B_{1}(0^{n})}|(v_{i})_{z_{i},r}(x)-v_{z_{i},r}(x)|dx+\int_{B_{1}(0^{n})}|v_{z_{i},r}(x)-v_{z,r}(x)|dx\\
  = &~~  \int_{B_{r}(z_{i})}r^{p-2-n}|v_{i}(x)-v(x)|dx+\int_{B_{r}(0^{n})}r^{p-2-n}|v(x+z_{i})-v(x+z)|dx\\
\rightarrow &~~ 0,
\end{split}
\end{equation*}
as desired.

\bigskip
\noindent
{\bf Case 2.} $p=2$.

\bigskip
By the definition of tangential $2$-flow, we have
\begin{equation*}
\begin{split}
& \int_{B_{1}(0^{n})}|(v_{i})_{z_{i},r}(x)-v_{z,r}(x)|dx\\
\leq & \int_{B_{1}(0^{n})}|v_{i}(rx+z_{i})-v(rx+z)|dx+\int_{B_{1}(0^{n})}|M(v_{i},z_{i},r)-M(v,z,r)|dx.
\end{split}
\end{equation*}
By the similar argument in Case 1, we obtain
\begin{equation*}
\lim_{i\rightarrow\infty}\int_{B_{1}(0^{n})}|v_{i}(rx+z_{i})-v(rx+z)|dx=0
\end{equation*}
Hence, it suffices to prove $\lim_{i\rightarrow\infty}M(v_{i},z_{i},r)=M(v,z,r)$. Next, we define $\tilde{v}_{i}(x)=v_{i}(x+z_{i}-z)$ for every $x\in B_{1}(0^{n})$. It then follows that $M(\tilde{v}_{i},z,r)=M(v_{i},z_{i},r)$. It is clear that
\begin{equation*}
\begin{split}
& \int_{B_{1}(0^{n})}|\tilde{v}_{i}(x)-v(x)|dx\\
\leq & \int_{B_{1}(0^{n})}|v_{i}(x+z_{i}-z)-v(x+z_{i}-z)|dx+\int_{B_{1}(0^{n})}|v(x+z_{i}-z)-v(x)|dx\\
   = & \int_{B_{1}(z_{i}-z)}|v_{i}(x)-v(x)|dx+\int_{B_{1}(0^{n})}|v(x+z_{i}-z)-v(x)|dx\\
\rightarrow & ~ 0,
\end{split}
\end{equation*}
where we used Lemma \ref{translation lemma}. Hence, by Lemma \ref{M S stable lemma}, we obtain $\lim_{i\rightarrow\infty}M(v_{i},z_{i},r)=\lim_{i\rightarrow\infty}M(\tilde{v}_{i},z,r)=M(v,z,r)$.
\end{proof}

~~~~~~


\noindent{Jianchun Chu}\\
School of Mathematical Sciences, Peking University\\
Yiheyuan Road 5, Beijing, P.R.China, 100871\\
Email:{chujianchun@pku.edu.cn}

\end{document}